\documentclass[a4paper,12pt,english]{amsart}

\usepackage{comment}

\usepackage{babel}
\usepackage{t1enc}
\usepackage[utf8]{inputenc}
\usepackage[T1]{fontenc}
\usepackage{amsmath,amssymb,amsfonts,amsthm}
\usepackage{graphicx}
\usepackage[colorlinks=true, allcolors=black]{hyperref}
\usepackage[a4paper, total={16 cm, 24 cm}]{geometry}
\usepackage{mathtools}
\usepackage{parallel,enumitem}
\usepackage[capitalise,nameinlink]{cleveref}
\usepackage{url}
\usepackage{lmodern}
\usepackage{domitian}
\usepackage{setspace}
 
\newcommand\restr[2]{{
\left.
\kern-
\nulldelimiterspace 
#1 
\right|_{#2} 
}}

\makeatletter
\def\keywords{\xdef\@thefnmark{}\@footnotetext}
\makeatother

\makeatletter
\def\subsection{\@startsection{subsection}{2}%
  \z@{.5\linespacing\@plus.7\linespacing}
{.5\baselineskip}%
  {\normalfont\centering\scshape}%
}
\makeatother

\DeclareMathOperator{\dist}{dist}

\DeclareMathOperator{\interior}{int}
\DeclareMathOperator{\identity}{id}
\DeclareMathOperator{\diam}{diam}

\DeclareMathOperator{\supp}{supp}
\DeclareMathOperator{\kernel}{Ker}

\DeclareMathOperator{\Ann}{Ann}

\makeatletter
\def\mathcenterto#1#2{\mathclap{\phantom{#1}\mathclap{#2}}\phantom{#1}}
\let\old@widetilde\widetilde
\def\widetildeto#1#2{\mathcenterto{#2}{\old@widetilde{\mathcenterto{#1}{#2\,}}}}
\let\old@widehat\widehat
\def\widehatto#1#2{\mathcenterto{#2}{\old@widehat{\mathcenterto{#1}{#2\,}}}}
\makeatother

\newcommand{\defeq}{\vcentcolon=}

\newcommand{\ds}{\displaystyle}

\newtheorem{theorem}{Theorem}[section]
\newtheorem*{thm:alg_c_word_p}{Theorem \ref{t.alg_c_word_p}}
\newtheorem{prop}[theorem]{Proposition}
\newtheorem{lemma}[theorem]{Lemma}

\newtheorem{cor}[theorem]{Corollary}

\theoremstyle{definition}
\newtheorem{defi}[theorem]{Definition}
\theoremstyle{definition}

\theoremstyle{definition}
\newtheorem{remark}[theorem]{Remark}
\theoremstyle{definition}
\newtheorem{notation}[theorem]{Notation}
\theoremstyle{definition}
\newtheorem{obs}[theorem]{Observation}
\theoremstyle{definition}

\theoremstyle{definition}
\newtheorem{question}[theorem]{Question}
\theoremstyle{definition}
\newtheorem{problem}[theorem]{Problem}

\def\nat{\ensuremath\mathbb{N}}
\def\int{\ensuremath\mathbb{Z}}
\def\rat{\ensuremath\mathbb{Q}}
\def\real{\ensuremath\mathbb{R}}

\def\TT{\ensuremath\mathbb{T}^\nat}
\def\circle{\ensuremath\mathbb{T}}

\def\eps{\ensuremath\varepsilon}
\def\cals{\ensuremath\mathcal{S}}
\def\caln{\ensuremath\mathcal{N}}

\def\calf{\ensuremath\mathcal{F}}
\def\calu{\ensuremath\mathcal{U}}
\def\calb{\ensuremath\mathcal{B}}

\def\calp{\ensuremath\mathcal{P}}
\def\calc{\ensuremath\mathcal{C}}

\def\calk{\ensuremath\mathcal{K}}
\def\calg{\ensuremath\mathcal{G}}
\def\cala{\ensuremath\mathcal{A}}
\def\cald{\ensuremath\mathcal{D}}
\def\cale{\ensuremath\mathcal{E}}
\def\cali{\ensuremath\mathcal{I}}

\def\calt{\ensuremath\mathcal{T}}

\def\nnn{\ensuremath\nat^{\nat\times\nat}}
\def\zdo{\ensuremath\mathbb{Z}^{(\nat)}}
\def\wtilde{\ensuremath\widetilde}
\def\what{\ensuremath\widehat}

\def\embeds{\ensuremath\hookrightarrow}

\def\su{\ensuremath\mathbb{SU}}
\def\hconv{\ensuremath\overset{d_H}{\longrightarrow}}
\def\nsub{\ensuremath\triangleleft}

\title{Generic properties of topological groups}

\author[M. Elekes]{M{\'a}rton Elekes}
\address{HUN-REN Alfréd Rényi Institute of Mathematics, Reáltanoda u. 13-15, 1053 Budapest, Hungary, AND ELTE Eötvös Loránd University, Budapest, Hungary,
URL: http://www.renyi.hu/$\sim$emarci, ORCID: 0000-0002-5139-2169}
\email{elekes.marton@renyi.hu}

\author[B. Geh{\'e}r]{Bogl{\'a}rka  Geh{\'e}r}
\address{ELTE Eötvös Loránd University, Budapest, Hungary}
\email{bogigeher@gmail.com}

\author[T. K{\'a}tay]{Tam{\'a}s K{\'a}tay}
\address{ELTE Eötvös Loránd University, Budapest, Hungary, ORCID: 0000-0002-7917-5887}
\email{13heted@gmail.com}

\author[T. Keleti]{Tam{\'a}s Keleti}
\address{ELTE Eötvös Loránd University, Budapest, Hungary, ORCID: 0000-0003-4849-5287}
\email{tamas.keleti@gmail.com}

\author[A. Kocsis]{Anett Kocsis}
\address{ELTE Eötvös Loránd University, Budapest, Hungary}
\email{sakkboszi@gmail.com}

\author[M. P{\'a}lfy]{M{\'a}t{\'e} P{\'a}lfy}
\address{ELTE Eötvös Loránd University, Budapest, Hungary}
\email{palfymateandras@gmail.com}

\begin{document}


\begin{abstract}
We study generic properties of topological groups in the sense of Baire category.

First, we investigate countably infinite groups. We extend a classical result of B.~H.~Neumann, H.~Simmons and A.~Macintyre on algebraically closed groups and the word problem. Recently, I.~Goldbring, S.~Kunnawalkam Elayavalli, and Y.~Lodha proved that every isomorphism class is meager among countably infinite groups. In contrast, it follows from the work of W.~Hodges on model-theoretic forcing that there exists a comeager isomorphism class among countably infinite abelian groups. We present a new elementary proof of this result.

Then we turn to compact metrizable abelian groups. We use Pontryagin duality to show that there is a comeager isomorphism class among compact metrizable abelian groups. We discuss its connections to the countably infinite case.

Finally, we study compact metrizable groups. We prove that the generic compact metrizable group is neither connected nor totally disconnected; also it is neither torsion-free nor a torsion group.
\end{abstract}

\maketitle


\newpage

\tableofcontents

\section{Introduction}

\textbf{Motivation.} Generic properties in the sense of Baire category have been extensively studied in numerous branches of mathematics. For example, there is a vast literature on the behavior of the generic continuous function.

More recently, E.~Akin, M.~Hurley, and J.~A.~Kennedy \cite{AKIN} investigated the dynamics of generic homeomorphisms of compact spaces. A.~Kechris and C.~Rosendal \cite{KECHRIS_ROSENDAL} studied the automorphism groups of homogeneous countable structures. Among numerous other results they characterized when an automorphism group admits a comeager conjugacy class.

Even more recently, M.~Doucha and M.~Malicki \cite{DOUCHA} examined generic representations of discrete countable groups in Polish groups. I.~Goldbring, S.~Kunnawalkam Elayavalli and Y.~Lodha \cite{GOLDBRING} studied generic algebraic properties in spaces of enumerated groups.
Among many other results, they proved that in the space $\calg'$ of all enumerated countably infinite groups the following hold:
\begin{itemize}
    \item every group property $\calp\subseteq\calg'$ with the Baire property is either meager or comeager \cite[Theorem~5.1.1]{GOLDBRING};
    \item every isomorphism class is meager in $\calg'$ \cite[Theorem~1.1.6 and Remark 1.1.4]{GOLDBRING};
    \item algebraically closed groups form a comeager set in $\calg'$ \cite[Lemma~5.2.7]{GOLDBRING}.
\end{itemize}
To maintain the coherence and clarity of the paper, in Subsection~3.1, we provide several definitions and basic observations that were essentially already present in \cite{GOLDBRING}. Also, we make heavy use of a special case of \cite[Lemma~5.2.7]{GOLDBRING}, which we state and prove as Theorem~\ref{t.alg_closed}.

\textbf{Setup and goals.} First, we extend the theory developed by I.~Goldbring, S.~Kunnawalkam Elayavalli, and Y.~Lodha: we study further generic properties of countably infinite groups. Then we turn to a more general family of problems: generic properties of topological groups. We investigate two important classes of topological groups:

In Section~\ref{s.cma_groups}, we study generic properties of compact metrizable abelian groups by equipping the set of compact subgroups of a universal compact metrizable abelian group $\TT$ with the Hausdorff metric. We denote this space by $\cals(\TT)$ and we view it as \textit{the space of compact metrizable abelian groups}. This setup was introduced in \cite{FISHER_GARTSIDE_1} and basic properties of spaces of the form $\cals(G)$ with compact $G$ were studied in \cite{FISHER_GARTSIDE_1} and \cite{FISHER_GARTSIDE_2}.

In Section~\ref{s.compact_metrizable_groups}, we use the same approach for compact metrizable groups (with the universal compact metrizable group $\prod_{n=1}^\infty SU(n)$).

\textbf{Main results and the organization of the paper.} Section~\ref{s.preliminaries} provides preliminaries in topology and algebra.

Section~\ref{s.countably_infinite_discrete_groups} focuses on two results. First, we extend a classical result of B.~H.~Neumann, H.~Simmons, and A.~Macintyre on algebraically closed groups and the word problem while we explore the generic subgroup structure of a countably infinite discrete group. Second, we give a new, elementary proof of the existence of a comeager isomorphism class among countably infinite discrete \emph{abelian} groups, which was previously known only by utilizing the work of W.~Hodges on model-theoretic forcing.


In Section~\ref{s.cma_groups}, we make heavy use of Pontryagin duality to prove that there is a comeager isomorphism class in the space of compact metrizable abelian groups. We discuss connections to the countable discrete abelian case.

In Section~\ref{s.compact_metrizable_groups}, we drop commutativity and study compact metrizable groups as compact subgroups of $\prod_{n=1}^\infty SU(n)$. Here we find a much less clear picture: the generic group proves to be heterogeneous both topologically and algebraically.

\section{Preliminaries}\label{s.preliminaries}

All topological groups are assumed to be Hausdorff.

\begin{notation}
For convenience let $\nat\defeq\{1,2,3,\ldots\}$ throughout this paper.
\end{notation} 

\begin{notation}\label{n.projection}
Let $(X_n)_{n\in \nat}$ be topological groups. Then let us define
$$\pi_{[n]}:\prod_{k\in\nat} X_k\to\prod_{k\in\nat} X_k,\qquad (x(1),x(2),\ldots)\mapsto(x(1),\ldots,x(n),1_{X_{n+1}},1_{X_{n+2}},\ldots).$$
and
$$\pi_n:\prod_{k\in\nat}X_k\to X_n,\qquad (x(1),x(2),\ldots)\mapsto x(n).$$
\end{notation}

Notice that the $\pi_n$ are surjective continuous homomorphisms and the range of $\pi_{[n]}$ is canonically isomorphic to $\prod_{k=1}^n X_k$.

\subsection{The space of compact subgroups}\label{ss.space_of_subgroups}

For a topological space $X$ let $\calk(X)$ denote the set of nonempty compact subsets of $X$. It is a topological space with the Vietoris topology, that is, the topology generated by basis elements of the form
\begin{equation}\tag{$\star$}
\{K\in \calk(X):\ K\subseteq U,\ K\cap V_1\neq\emptyset,\ldots,K\cap V_n\neq\emptyset\}
\end{equation}
with $U,V_1,\ldots,V_n$ open in $X$.

If $(X,d)$ is a metric space, then $\calk(X)$ is also a metric space with the Hausdorff metric:
$$d_H(K,L)\defeq \inf\{\eps>0:\ K_\eps\supseteq L,\ L_\eps\supseteq K\},$$
where $A_\eps$ is the $\eps$-neighborhood of the set $A$. It is well-known that the Hausdorff metric induces the Vietoris topology. It is also well-known that $(\calk(X),d_H)$ inherits several topological properties of $(X,d)$. For example, if $X$ is compact, then $\calk(X)$ is also compact (see \cite[Subsection~4.F]{KECHRIS}).

We will need the following well-known lemma.

\begin{lemma}\label{l.continuous_image}
Let $X$, $Y$ be topological spaces. A continuous map $f:X\to Y$ induces a continuous map:
$$\calk(f):\calk(X)\to \calk(Y),\quad K\mapsto f(K).$$
\end{lemma}

\begin{proof}
It is easy to check that inverse images of basis elements of the form ($\star$) are open.
\end{proof}

Now we turn to topological groups.

\begin{prop}
For any topological group $G$ the set
$$\cals(G)\defeq\{K\in \calk(G):\ K\text{ is a subgroup of }G\}$$
is closed in $\calk(G)$.
\end{prop}
\begin{proof}
Observe that
$$\cals(G)=\{K\in\calk(G):\ K\cdot K^{-1}=K\}.$$
It is easy to check that this set is closed. (Use Lemma~\ref{l.continuous_image}.)
\end{proof}

\begin{defi}\label{d.compact_subgroups}
We call the above defined $\cals(G)$ \textbf{the space of compact subgroups} of $G$. Note that for a compact metrizable group $G$ the space $\cals(G)$ is also compact metrizable.
\end{defi}

We need the following well-known fact, which is a special case of \cite[Theorem~1.6]{OPEN_MAP}.

\begin{theorem}\label{t.surj_cont_hom_open}
Let $G$ and $H$ be compact groups. If $\varphi:G\to H$ is a surjective continuous homomorphism, then it is open.
\end{theorem}

\subsection{Baire category}\label{ss.baire_category}

In this subsection we present well-known theorems and notions. All of them can be found in \cite{KECHRIS}.

A topological space $(X,\tau)$ is called \textbf{Polish} if it is separable and completely metrizable. Clearly, countable discrete spaces are Polish. In particular, $2=\{0,1\}$ and $\nat$ with the discrete topologies are Polish. The following proposition is well-known:

\begin{prop}\label{p.prod_polish}
Countable products of Polish spaces are Polish. In particular, $2^A$ and $\nat^A$ are Polish for any countable set $A$.
\end{prop}

A subset $E$ of a topological space $X$ is called \textbf{nowhere dense} if the closure of $E$ has empty interior, \textbf{meager} if it is a countable union of nowhere dense sets, and \textbf{comeager} if its complement is meager.

\begin{theorem}[Baire Category Theorem]\label{t.bct}
In a completely metrizable space every nonempty open set is nonmeager.
\end{theorem}

If $X$ is a completely metrizable space of objects (sets, functions, groups, etc.), we say that the property $P$ is \textbf{generic} in $X$ if $\{x\in X:\ x\text{ is of property }P\}$ is comeager in $X$. We also say that the generic $x\in X$ has property $P$.

Another well-known theorem characterizes Polish subspaces of Polish spaces.

\begin{theorem}\label{t.G_delta}
A subspace $E$ of a Polish space $X$ is Polish if and only if $E$ is $G_\delta$ in $X$. In particular, closed subspaces are Polish.
\end{theorem}

\begin{remark}\label{r.baire_category}
It is well-known that a subset of a Polish space is comeager if and only if it contains a dense $G_\delta$ set. It follows immediately that an $F_\sigma$ subset of a Polish space is nonmeager if and only if it has nonempty interior.
\end{remark}


\subsection{Group theory}\label{ss.algebra}

We will need the following notion from combinatorial group theory.

\begin{defi}
A finitely generated group $G$ has \textbf{solvable word problem} if there exists a Turing machine that decides for every word in the generators of $G$ whether it represents the identity element. It is easy to see that the solvability of the word problem is independent of the choice of the finite generating set.
\end{defi}

Let $F$ be a free group on the infinite generating set $X=\{x_1,x_2,\dots\}$ and $G$ be any group. Recall (see \cite[Chapter~3]{MASSEY}) that an element of the free product $F\star G$ is a word whose letters are from $X$ and $G$. Intuitively, in a word $w\in F\star G$ letters from $X$ are variables and letters from $G$ are parameters.

Let $E$ and $I$ be finite subsets of $F\star G$. We view $E$ as a set of \textbf{equations} and $I$ as a set of \textbf{inequations}. A \textbf{solution} of the system $(E,I)$ in $G$ is a homomorphism $f: F\to G$ such that the unique homomorphism $\wtilde f: F\star G\to G$ extending both $f$ and the identity of $G$ maps every $e\in E$ to $1_G$ and does not map any $i\in I$ to $1_G$.

The system $(E,I)$ is \textbf{consistent with $G$} if it has a solution in some larger group $H\geq G$. That is, if there exists a group $H$
and an embedding $h:G\to H$ such that for the unique 
homomorphism $\wtilde h:F\star G\to F\star H$ extending both $h$ and the identity of $F$, the system $\left(\wtilde h (E),\wtilde h(I)\right)$ has a solution in $H$.

\begin{defi}\label{d.alg_closed}
The group $G$ is \textbf{algebraically closed} if every finite system $(E,I)$ of equations and inequations that is consistent with $G$ has a solution in $G$. 
\end{defi}

\begin{remark}
It follows immediately from the definition that every algebraically closed group is infinite. Some authors prefer the term existentially closed that comes from model theory and they reserve the term algebraically closed for the case when one does not allow inequations in the definition. However, B.~H.~Neumann proved in \cite{NEUMANN2} that the two notions coincide except for the trivial group. Since we study only infinite groups, the terminology will cause no confusion.
\end{remark}

By a standard closure argument, one easily verifies the following.

\begin{theorem}[{Scott, \cite[Theorem~1]{SCOTT}}]\label{t.emb_alg_closed}
Every countable group can be embedded into a countable algebraically closed group. In particular, algebraically closed groups exist.
\end{theorem}

We also need the following proposition.

\begin{prop}\label{p.count_inf_alg_c_sgp}
Every algebraically closed group has a countably infinite algebraically closed subgroup.
\end{prop}

To prove this we need two theorems.

\begin{theorem}\label{t.amalgamated_free_product}\cite[Chapter~III]{NEUMANN3}
If $H$ and $K$ are groups with a common subgroup $G=H\cap K$, then there exists a group $M$ such that $H$ and $K$ are subgroups of $M$.
\end{theorem}

\begin{defi}\label{d.elementary_subgroup}
In a group $G$ a subgroup $H$ is \textbf{elementary} if for every first-order formula $\varphi$ in the language of group theory with parameters from $H$ we have: $\varphi$ holds in $G$ if and only if $\varphi$ holds in $H$.
\end{defi}

By the Downward Löwenheim-Skolem Theorem (see \cite[Corollary~3.1.5]{HODGES}) we have the following.

\begin{theorem}\label{t.lowenheim_skolem}
If $G$ is a group, then for every cardinality $\aleph_0\leq\kappa\leq |G|$ there is an elementary subgroup of $G$ with cardinality $\kappa$.
\end{theorem}

\begin{proof}[Proof of Proposition~\ref{p.count_inf_alg_c_sgp}]
Fix an algebraically closed group $H$. By Theorem~\ref{t.lowenheim_skolem} there is a countably infinite elementary subgroup $G$ of $H$. We claim that $G$ is algebraically closed.

Pick any finite system $(E,I)$ of equations and inequations with parameters from $G$ such that $(E,I)$ has a solution in some bigger group $K\geq G$. Clearly, we may choose $K$ so that $G=H\cap K$. By Theorem~\ref{t.amalgamated_free_product} there is a group $M$ such that $H,K\leq M$. Clearly, $(E,I)$ has a solution in $M$ but contains only parameters from $G$. Since $H$ is algebraically closed, $(E,I)$ has a solution in $H$ as well. Moreover, $(E,I)$ has a solution in $G$ because $G$ is an elementary subgroup of $H$, which concludes the proof.
\end{proof}

\begin{defi}\label{d.homogeneous}
A group $G$ is \textbf{homogeneous} if any isomorphism between two finitely generated subgroups of $G$ extends to an automorphism of $G$. If this extension can always be chosen to be an inner automorphism, then $G$ is \textbf{strongly homogeneous}.
\end{defi}

The following lemma is an instance of a well-known fact in model theory (see, for example, \cite[Lemma~7.1.3]{HODGES}). 

\begin{lemma}\label{l.fin_gen_extends}
Let $G$ be a countable group and $H$ be a homogeneous group. Then $G$ can be embedded into $H$ if and only if every finitely generated subgroup of $G$ can be embedded into $H$.
\end{lemma}




The following proposition is based on the existence of so-called HNN extensions, a celebrated result of Higman, Neumann, and Neumann. For the proof, see \cite[Lemma~1]{MACINTYRE1}.

\begin{prop}[Macintyre]\label{p.alg_closed_strongly_hom}
Algebraically closed groups are strongly homogeneous.
\end{prop}




Finally, we cite a very nice result that connects the word problem and algebraically closed groups. It is due to B.~H.~Neumann, H.~Simmons and A.~Macintyre, see \cite[Corollary~3.5]{NEUMANN1}, \cite[Theorem~C]{SIMMONS} and \cite[Corollary~2]{MACINTYRE2}.

\begin{theorem}[Macintyre--Neumann--Simmons]\label{t.alg_c_word_p}
A finitely generated group has solvable word problem if and only if it is embeddable into every algebraically closed group.
\end{theorem}

\begin{remark}\label{r.prop_of_ac}
It is well-known that algebraically closed groups are simple \cite{NEUMANN2} and not finitely generated \cite[Theorem~2.1]{NEUMANN1}. Since they contain free subgroups by Theorem~\ref{t.alg_c_word_p}, they are not locally finite either (recall that a group $G$ is locally finite if every finitely generated subgroup of $G$ is finite).
\end{remark}

\section{Countably infinite discrete groups}\label{s.countably_infinite_discrete_groups}

\subsection{The space of multiplication tables}\label{ss.multip_tables}

By Proposition~\ref{p.prod_polish} the space $\nnn$ of infinite tables of natural numbers is Polish. A clopen basis for $\nnn$ consists of sets of the form
\begin{equation}\label{eq.basis1}
    \left\{A\in\nnn:\ A(n_1,m_1)=k_1,\ldots,A(n_l,m_l)=k_l\right\}
\end{equation}
for $n_i,m_i,k_i\in\nat$, ($i=1,\ldots,l$). We will study the subspace
$$\calg\defeq\left\{A\in\nnn:\ A\text{ is the multiplication table of a group and $1$ is its identity element}\right\}.$$

\begin{defi}
A \textbf{group property} is an isomorphism-invariant subset $\calp\subseteq\calg$. The property $\calp$ is \textbf{generic} if it is a comeager subset of $\calg$.
\end{defi}
We would like to apply the Baire category theorem (Theorem~\ref{t.bct}) in $\calg$. Therefore, by Theorem~\ref{t.G_delta} we need to prove that $\calg$ is $G_\delta$ in $\nnn$. The calculation is quite straightforward, see \cite[Proposition~3.1]{ELEKES}.

\begin{remark}\label{r.space_goldbring}
Consider the space introduced in \cite[Section~3.1]{GOLDBRING}, which is also denoted by $\calg$. To avoid confusion, here we denote it by $\calg'$. There is a natural embedding $\Phi:\calg\embeds\calg'$ that maps $\calg$ onto a clopen subset of $\calg'$: multiplication tables determine inverses and we take $1$ to be the identity element. Moreover, $\calg'$ is the disjoint union of countably infinitely many copies of $\Phi(\calg)$. Therefore, regarding the genericity of group properties, $\calg$ and $\calg'$ are equivalent. Also, it is pointed out in \cite[Proposition~3.1.1]{GOLDBRING} that the subspace $\calg'$ is closed in $\nnn\times\nat^\nat\times\nat$. On the other hand, $\calg$ is not closed in $\nnn$.
\end{remark}

We reserve the notation $\wtilde G$ for the unique group of multiplication table $G$ and underlying set $\nat$.

\begin{remark}
When we consider elements of $\calg$ we use the usual shorthands of group theory. For example, to define the subspace of torsion groups we write
$$\{G\in\calg:\ \forall n\in\nat\ \exists k\in\nat\ n^k=1\}$$
instead of the rather cumbersome
$$\{G\in\calg:\ \forall n\in\nat\ \exists k\in\nat\ \underbrace {G(G(\ldots G(G(n,n),n),\ldots, n),n)}_{k\text{ times}}=1\}.$$
\end{remark}

\begin{remark}\label{r.inverse}
We also use inverses for convenience:

Fix any $n,m,k\in\nat$. Then e.g. $nm^{-1}=k$ abbreviates $\forall x\ (mx=1\implies nx=k)$. Clearly, this defines a closed subset of $\calg$. However, the group axioms imply that it is equivalent to $\exists x\ (mx=1\land nx=k)$, which defines an open set. Thus $nm^{-1}=k$ defines a clopen set.
\end{remark} 

\begin{obs}\label{o.basis1}
The above-mentioned standard clopen basis (\ref{eq.basis1}) of $\nnn$ induces a clopen basis for $\calg$: sets of the form
\begin{equation}\label{e.basis1}
    \{G\in\calg:\ \forall i\leq l\ (n_i\cdot m_i=k_i)\}
\end{equation}
with $l\in\nat$ and $n_i,m_i,k_i\in\nat$ for all $i\leq l$ constitute a basis.
For practical reasons we introduce another clopen basis for $\calg$.
\end{obs}

\begin{prop}\label{p.basis2}
For any finite sets $\{U_1,\dots,U_k\}$ and $\{V_1,\dots,V_l\}$ of words in $n$ variables $x_1,\dots,x_n$ and for any $a_1,\dots,a_n,b_1,\dots,b_k,c_1,\dots,c_l\in\nat$ the set
\begin{equation}
    \left\{G\in\calg:\ \bigwedge_{i=1}^k U_{i}(a_1,\dots,a_n)=b_i\land \bigwedge_{j=1}^l V_{j}(a_1,\dots,a_n)\neq c_j\right\}
\end{equation}
is clopen, and sets of this form constitute a basis for $\calg$.
\end{prop}

\begin{proof}
This family extends (\ref{e.basis1}), so it suffices to show that its elements are clopen. Since clopen sets form an algebra, it suffices to consider only one word $W$ and only the case of equation. The proof is a straightforward induction on the length of $W$ and uses the same argument as Remark~\ref{r.inverse}.
\end{proof}

\begin{obs}\label{o.induced_homeo}
An important observation is that permutations induce homeomorphisms. Let $\varphi:\nat\to\nat$ be a bijection that fixes 1. Then the \emph{induced homeomorphism} $h_\varphi:\calg\to\calg$ is defined as follows. Intuitively, we define $h_\varphi(G)$ by pushing forward the structure of $G$ via $\varphi$. More precisely, for any  $G\in\calg$ the multiplication table $h_\varphi(G)$ is defined by the equations $i\cdot j\defeq\varphi(G(\varphi^{-1}(i),\varphi^{-1}(j)))$ for all $i,j\in\nat$. Thus $\varphi$ is an isomorphism between $\wtilde G$ and $\wtilde{h_\varphi(G)}$. It is an easy exercise to verify that $h_\varphi$ is indeed a homeomorphism.
\end{obs}

\subsection{Subgroups}

As we have already mentioned in Subsection~\ref{ss.algebra}, the following nice and highly nontrivial result is due to B.~H.~Neumann, H.~Simmons and A.~Macintyre, see \cite[Corollary~3.5]{NEUMANN1}, \cite[Theorem~C]{SIMMONS} and \cite[Corollary~2]{MACINTYRE2}.

\begin{thm:alg_c_word_p}[Macintyre--Neumann--Simmons]
A finitely generated group has solvable word problem if and only if it is embeddable into every algebraically closed group.
\end{thm:alg_c_word_p}

From Theorem~\ref{t.alg_c_word_p} it is not hard to prove the following corollary. Here we omit its proof because it follows from Theorem~\ref{t.embed}, which we prove in this subsection.

\begin{cor}
The following are equivalent:

(1) The group $H$ can be embedded into every algebraically closed group.

(2) The group $H$ is countable and every finitely generated subgroup of $H$ has solvable word problem.
\end{cor}

Note that these results connect algebraic closedness, which is a purely algebraic notion, to computability theory. One of the main goals of this section is to prove that the purely topological property introduced in the following definition is equivalent to (1) and (2). We find this phenomenon noteworthy since it connects three fields of mathematics.

\begin{defi}\label{d.gen_emb}
For a group $H$ let us define
$$\cale_H\defeq\left\{G\in\calg:\ H \text{ can be embedded into } \wtilde G\right\}.$$
We say that the group $H$ is \textbf{generically embeddable} if $\cale_H$ is comeager.
\end{defi}

To prove the above-mentioned equivalence (Theorem~\ref{t.embed}) we need several results, many of which are interesting in their own.

\begin{notation}\label{n.supp}
If $\calb$ is a basic clopen set of the form
$\{G\in\calg:\ \forall i,j\leq k\ (i\cdot j=m_{i,j})\}$ with $k\in\nat$ and $m_{i,j}\in\nat$ for all $i,j\leq k$, then let $\supp\calb\defeq\{1,\dots,k\}\cup\{m_{i,j}:\ i,j\leq k\}$.
\end{notation}

The following theorem is a special case of \cite[Lemma~5.2.7]{GOLDBRING} (recall Remark~\ref{r.space_goldbring}). For the sake of completeness and understandability we present our own proof.

\begin{theorem}[Goldbring--Kunnawalkam Elayavalli--Lodha]\label{t.alg_closed}
The set
$$\calc\defeq\{G\in\calg:\ \wtilde G\text{ is algebraically closed}\}$$
is comeager in $\calg$.
\end{theorem}

\begin{proof}
Let $F$ be the free group generated by $X=\{x_1,x_2\dots\}$. Let $G\in\calg$. By definition, $\wtilde G$ is algebraically closed if and only if the following holds. For every finite system $(E,I)$ of equations and inequations with elements from $F\star\wtilde G$ either $(E,I)$ is inconsistent with $\wtilde G$ or it has a solution in $\wtilde G$ (see Subsection~\ref{ss.algebra} for the definitions). Note that elements of $F\star \wtilde G$ are words with letters from $X$ and $\nat$. Clearly, there are countably many such words. Thus it suffices to prove that for any finite system $(E,I)$ the set
$$\calc(E,I)\defeq\{G\in\calg:\ (E,I)\text{ is inconsistent with } \wtilde G\text{ or $(E,I)$ has a solution in } \wtilde G\}$$
contains a dense open set. Fix
$$E=\{U_1(x_1,\dots,x_n),\dots,U_k(x_1,\dots,x_n)\}\quad \text{and}\quad I=\{V_1(x_1,\dots,x_n),\dots,V_l(x_1,\dots,x_n)\}$$
\noindent where $x_1,\dots,x_n$ are the variables occurring in elements of $E$ and $I$.

Fix a nonempty basic clopen set $\calb=\{G\in\calg:\ \forall i,j\leq k\ (i\cdot j=m_{i,j})\}$ with $k\in\nat$ and $m_{i,j}\in\nat$ for all $i,j\leq k$. We need to show that $\calb\cap\calc(E,I)$ contains a nonempty open set. There are two cases.

\textbf{Case 1.} There exists $H\in\calb$ such that $(E,I)$ has a solution in $\wtilde H$. That is, for some natural numbers $a_1,\dots,a_n$ we have $\bigwedge_{i=1}^k U_i(a_1,\dots,a_n)=1$ and $\bigwedge_{j=1}^l V_j(a_1,\dots,a_n)\neq 1$ in $\wtilde H$. Then
$$\calu\defeq\left\{G\in\calg:\ \bigwedge_{i=1}^k U_i(a_1,\dots,a_n)=1\land \bigwedge_{j=1}^l V_j(a_1,\dots,a_n)\neq 1\right\}$$
is clopen in $\calg$ by Proposition~\ref{p.basis2}. Now $\calu\subseteq\calc(E,I)$ and $H\in\calu\cap\calb$, hence $\calu\cap\calb$ is a nonempty open subset of $\calb\cap\calc(E,I)$, which completes Case 1.

\textbf{Case 2.} For every $G\in\calb$ the finite system $(E,I)$ is unsolvable in $\wtilde G$. It suffices to prove that for every $G\in\calb$ the finite system $(E,I)$ is inconsistent with $\wtilde G$.

Suppose that there is some $H\in\calb$ and a group $L\geq \wtilde H$ such that $(E,I)$ has a solution $a_1,\dots,a_n$ in $L$. Clearly, we may assume that $L$ is countable since otherwise we could replace it by one of its countably generated subgroups. Let $M$ be the set of natural numbers occurring in elements of $E$ or $I$. We choose any bijection $\varphi:L\to \nat$ that extends the identity of the finite set $M\cup(\supp\calb)\cup\{1\}$ (recall Notation~\ref{n.supp}). We define the multiplication table $K$ on $\nat$ by pushing forward the structure of $L$ via $\varphi$. Then $\wtilde K\in\calb$ and $\varphi(a_1),\dots,\varphi(a_n)$ is a solution of $(E,I)$ in $\wtilde K$, a contradiction.
\end{proof}

\begin{cor}\label{c.generic_simple}
The generic $G\in\calg$ is simple, not finitely generated and not locally finite.
\end{cor}

\begin{proof}
This follows from Remark~\ref{r.prop_of_ac} and Theorem~\ref{t.alg_closed}.
\end{proof}

Recall from logic that an existential sentence is a closed formula that starts with a string of existential quantifiers followed by only a quantifier-free formula.

\begin{remark}\label{r.ec_ac}
It is well-known \cite[page 47]{HODGES_1985} that a group $G$ is algebraically closed if and only if the following holds: for every group $H\geq G$ and existential formula $\varphi$ with parameters from $G$, we have that whenever $\varphi$ holds in $H$, it holds in $G$ as well.
\end{remark}

\begin{lemma}\label{l.sol_in_all_alg_c}
Let $\Phi$ be an existential sentence in the first-order language of group theory. If $\Phi$ holds in some group, then $\Phi$ holds in every algebraically closed group.
\end{lemma}

\begin{proof}
Let $H$ be a group such that $\Phi$ holds in $H$. By the definition of algebraically closed groups and Remark~\ref{r.ec_ac}, it suffices to show that for any algebraically closed group $G$ the sentence $\Phi$ holds in $G\times H$. Since $\{1\}\times H$ and $H$ are isomorphic, $\Phi$ holds in $\{1\}\times H$. Hence $\Phi$ holds in $G\times H$ because $\Phi$ is existential.
\end{proof}

\begin{theorem}\label{t.alg_isom_dense}
The isomorphism class of any countably infinite algebraically closed group is dense in $\calg$.
\end{theorem}

\begin{proof}
Fix a nonempty basic clopen $\calb=\{G\in\calg:\ \forall i,j\leq k\ (i\cdot j=m_{i,j})\}$ with $k\in\nat$ and $m_{i,j}\in\nat$ for all $i,j\leq k$. Fix any multiplication table $A\in\calg$ with $\wtilde A$ algebraically closed. Pick any $H\in\calb$. We associate variables $x_i$ to $i$ for each $i\leq k$ and $x_{i,j}$ to $(i,j)$ for each $i,j\leq k$. For each $i,j,l\leq k$ let
$$
\varphi_{i,j,l}=
\begin{cases}
x_{i,j}=x_l\qquad\text {if}\quad m_{i,j}=l, \\
x_{i,j}\neq x_l\qquad\text{if}\quad m_{i,j}\neq l,
\end{cases}
$$
and for each $i,j,r,s\leq k$ let
$$
\varphi_{i,j,r,s}=
\begin{cases}
x_{i,j}=x_{r,s}\qquad\text {if}\quad m_{i,j}=m_{r,s}, \\
x_{i,j}\neq x_{r,s}\qquad\text{if}\quad m_{i,j}\neq m_{r,s}.
\end{cases}
$$
Now let $\Phi$ denote the following formula:
\begin{equation*}
\begin{array}{cc}
\exists x_1,\ldots,x_k,x_{1,1},\ldots,x_{1,k},x_{2,1},\ldots,x_{2,k},\ldots,x_{k,1},\ldots,x_{k,k}  \\
\ds\left(\left(\bigwedge_{i,j\leq k} x_i\cdot x_j=x_{i,j}\right)\land\left(\bigwedge_{\substack{i,j\leq k\\ i\neq j}}x_i\neq x_j\right)\land\left(\bigwedge_{i,j,l\leq k}\varphi_{i,j,l}\right)\land\left(\bigwedge_{i,j,r,s\leq k}\varphi_{i,j,r,s}\right)\right).
\end{array}
\end{equation*}

Clearly, $\Phi$ holds in $\wtilde H$. Thus $\Phi$ holds in $\wtilde A$ as well by Lemma~\ref{l.sol_in_all_alg_c}. That is, there are numbers $a_i, n_{i,j}\in \nat$ such that the equations $a_i\cdot a_j=n_{i,j}$ hold in $\wtilde A$ for each $i,j\leq k$ and two of them equal if and only if the corresponding elements are equal in $H$. Note that $a_1=1$ since both $a_1\cdot a_1=n_{1,1}$ and $n_{1,1}=a_1$ hold. Let $\alpha:\nat\to\nat$ be a bijection that extends the finite map $a_i\mapsto i$, $n_{i,j}\mapsto m_{i,j}$ for each $i,j\leq k$. Then $h_\alpha(A)\in\calb$ and $\wtilde{h_\alpha(A)}$ is isomorphic to $\wtilde A$ (recall from Observation~\ref{o.induced_homeo} that $h_\alpha$ is the induced homeomorphism). 
\end{proof}

\begin{lemma}\label{l.F_sigma}
For any finitely generated group $H$ the set $\cale_H$ is $F_\sigma$.
\end{lemma}

\begin{proof}
Fix a finite generating set $\{a_1,\dots,a_n\}$ for $H$. Let $G\in\calg$ be arbitrary.

\textbf{Claim.} The group $H$ is embeddable into $\wtilde G$ if and only if there are $b_1,\dots,b_n\in\nat$ such that for every word $W$ in $n$ variables $W(a_1,\dots,a_n)=1$ in $H$ $\iff$ $W(b_1,\dots,b_n)=1$ in $\wtilde G$.

If there is an embedding $i:H\embeds \wtilde G$, then clearly for every word $W$ in $n$ variables $W(a_1,\dots,a_n)$ maps to $W(i(a_1),\dots,i(a_n))$, therefore $b_j=i(a_j)$ for each $1\leq j\leq n$ is suitable. On the other hand, if for some $b_1,\dots,b_n\in \wtilde G$ we have $W(a_1,\dots,a_n)=1$ in $H$ $\iff$ $W(b_1,\dots,b_n)=1$ in $\wtilde G$ for every word $W$ in $n$ variables, then $H$ and $\langle b_1,\dots,b_n\rangle_{\wtilde G}$ have the same presentation, hence they are isomorphic. This proves the claim.

Thus we may write $\cale_H$ as
$$\bigcup_{b_1,\dots,b_n\in\nat}\ \bigcap_{\substack{W\text{ is a word}\\ \text{in }n\text{ variables}}}\underbrace{\{G\in\calg:\ W(a_1,\dots,a_n)=1\text{ in }H \iff W(b_1,\dots,b_n)=1\text{ in }\wtilde G\}}_{\text{clopen by Proposition }\ref{p.basis2}},$$
which proves the lemma.
\end{proof}

Now we turn to the main result of this subsection.

\begin{theorem}\label{t.embed}
The following are equivalent:

(1) The group $H$ is generically embeddable.

(2) The group $H$ can be embedded into every algebraically closed group.

(3) The group $H$ is countable and every finitely generated subgroup of $H$ has solvable word problem.
\end{theorem}

\begin{proof}
(2)$\implies$(3) follows from Theorem~\ref{t.emb_alg_closed} and Theorem~\ref{t.alg_c_word_p}.

(3)$\implies$(2): Suppose (3). By Theorem~\ref{t.alg_c_word_p} every finitely generated subgroup of $H$ can be embedded into every algebraically closed group. Since every algebraically closed group is homogeneous by Proposition~\ref{p.alg_closed_strongly_hom}, (2) follows from Lemma~\ref{l.fin_gen_extends}.

(2)$\implies$(1) is immediate from Theorem~\ref{t.alg_closed}.

(1)$\implies$(2): Fix a generically embeddable group $H$. By Lemma~\ref{l.fin_gen_extends} it suffices to prove that any finitely generated subgroup $K$ of $H$ can be embedded into every algebraically closed group. Fix $K$. Note that $\cale_K$ is comeager because it contains $\cale_H$. On the other hand, $\cale_K$ is $F_\sigma$ by Lemma~\ref{l.F_sigma}, hence it has nonempty interior by Remark~\ref{r.baire_category}. However, isomorphism classes of algebraically closed groups are dense in $\calg$ by Theorem~\ref{t.alg_isom_dense}; therefore $K$ can be embedded into every countably infinite algebraically closed group. Now Proposition~\ref{p.count_inf_alg_c_sgp} completes the proof.
\end{proof}

\begin{remark}
We note that (3)$\implies$(1) follows easily from \cite[Theorem~1.1.3 and Remark 1.1.4]{GOLDBRING}.
\end{remark}

\subsection{Abelian groups}\label{ss.abelian_case}

In this section, we give a new, elementary proof of the existence of a comeager isomorphism class among countably infinite discrete \emph{abelian} groups, that is, a comeager isomorphism class in the subspace $\cala\defeq\{G\in\calg:\ \wtilde G\text{ is abelian}\}$. This was previously known only by utilizing the work of W.~Hodges on model-theoretic forcing.


To avoid confusion and preserve the coherence of Section~\ref{s.countably_infinite_discrete_groups} we \textbf{do not} switch to additive notation in elements of $\cala$.

First of all, observe that
$$\cala=\{G\in\calg:\ \forall n,k\ (nkn^{-1}k^{-1}=1)\}=\bigcap_{n,k\in\nat}\{G\in\calg:\ nkn^{-1}k^{-1}=1\}$$
is a closed subspace of $\calg$, therefore it is a Polish space.

\begin{remark}
It is very easy to see that $\cala\subseteq\calg$ has empty interior, hence it is nowhere dense. Thus for a comeager property $\calp\subseteq\calg$ the set $\calp\cap\cala$ may be meager in $\cala$.
\end{remark}

The following theorem is well-known, see \cite[Chapter~4, Theorem~3.1]{FUCHS}.

\begin{theorem}\label{t.divisible}
Every divisible abelian group is of the form
$$\left(\bigoplus_{p\in\mathbf{P}}\int[p^\infty]^{(I_p)}\right)\oplus\rat^{(I)},$$
where $\mathbf{P}$ is the set of prime numbers, $\int[p^\infty]$ is the well-known Prüfer $p$-group, $I$ and $I_p$ are arbitrary sets of indices, and for a group $G$ and a set $J$ the term $G^{(J)}$ abbreviates the direct sum $\bigoplus_{j\in J}G$.
\end{theorem}

Let $A\in\calg$ be such that
$$\wtilde A\cong\bigoplus_{p\in\mathbf{P}}\int[p^\infty]^{(\nat)}.$$
\begin{remark}\label{r.abelian_embeds}
Another well-known theorem is that every abelian group can be embedded into a divisible abelian group, see \cite[Chapter~4, Theorem~1.4]{FUCHS}. Thus, by Theorem~\ref{t.divisible}, every countable abelian torsion group can be embedded into $\wtilde A$.
\end{remark}

\begin{remark}\label{r.unique_abelian}
If $G$ is a divisible abelian torsion group, then it can be written as the torsion summand in Theorem~\ref{t.divisible}. If every finite abelian group can be embedded into $G$, then $I_p$ is infinite for every $p\in\mathbf{P}$ since $G$ must contain infinitely many elements of order $p$ for every $p$. If $G$ is countable, then $I_p$ is countable for every $p\in\mathbf{P}$. Therefore, $\wtilde A$ is the unique, up to isomorphism, countable, divisible abelian torsion group that contains every finite abelian group (up to isomorphism).
\end{remark}

\begin{prop}\label{p.abel_G_delta}
The sets
$$\cald\defeq\{G\in\cala:\ \wtilde G\text{ is divisible}\},\qquad \calt\defeq\{G\in\cala:\ \wtilde G\text{ is a torsion group}\}$$ and
$$\calf\defeq\{G\in\cala:\ \text{every finite abelian group can be embedded into }\wtilde G\}$$ are $G_\delta$ in $\cala$.
\end{prop}

\begin{proof}
By definition,
$$\cald=\{G\in\cala:\ \forall n,k\in\nat\ \exists m\in\nat\ (m^k=n)\}=\bigcap_{n,k\in\nat}\ \bigcup_{m\in\nat}\ \underbrace{\{G\in\cala:\ m^k=n\}}_{\text{clopen}},$$
which is $G_\delta$ in $\cala$.
Again, by definition,
$$\calt=\{G\in\cala:\ \forall n\in\nat\ \exists k\in\nat\ (n^k=1)\}=\bigcap_{n\in\nat}\ \bigcup_{k\in\nat}\ \underbrace{\{G\in\cala:\ n^k=1\}}_{\text{clopen}},$$
which is $G_\delta$ in $\cala$.
For $\calf$ note that there are countably infinitely many finite abelian groups up to isomorphism, hence it suffices to prove that for any fixed finite abelian group $H$ the set $\cala\cap\cale_H$ is $G_\delta$. Let $\{h_1,\dots,h_n\}$ be the underlying set of $H$. For a $G\in\cala$ the group $H$ is embeddable into $\wtilde G$ if and only if there exist pairwise distinct numbers $g_1,\dots,g_n\in \nat$ such that $h_i\cdot h_j=h_l$ in $H$ implies $g_i\cdot g_j=g_l$ in $\wtilde G$ for all $i,j,l\leq n$. That is,
$$\cale_H\cap\cala=\bigcup_{\substack{g_1,\dots,g_n\in\nat\\ \text{p.~distinct}}}\ \bigcap_{\substack{h_i,h_j,h_l\in H,\\ h_i\cdot h_j=h_l}}\ \{G\in\cala:\ g_i\cdot g_j=g_l\},$$
which is open in $\cala$.
\end{proof}

From Remark~\ref{r.unique_abelian} we know that the isomorphism class $\cali_A\defeq\{G\in\calg:\ \wtilde G\cong \wtilde A\}$ can be written as $\cala\cap\cald\cap\calt\cap\calf$, which is $G_\delta$ in $\cala$. Now we show that it is dense in $\cala$.

\begin{prop}\label{p.abel_dense}
The isomorphism class $\cali_A$ is dense in $\cala$.
\end{prop}

\begin{proof}
Fix any nonempty basic clopen $\calb=\{G\in\cala:\ \forall i,j\leq k\ (i\cdot j=m_{i,j})\}$ with $k\in\nat$ and $m_{i,j}\in\nat$ for all $i,j\leq k$. Pick some $H\in\calb$ and a divisible abelian group $K$ with an embedding $\varphi:\wtilde H\embeds K$; such a $K$ exists by Remark~\ref{r.abelian_embeds}. Clearly, $K$ can be chosen to be countable. We write $K$ in the form
$$\left(\bigoplus_{p\in\mathbf{P}}\int[p^\infty]^{(I_{K,p})}\right)\oplus\rat^{(I_K)}.$$
Consider the finite subset $\varphi(\supp\calb)\subseteq K$ (recall Notation~\ref{n.supp}). Every $x\in\varphi(\supp\calb)$ has finitely many nonzero coordinates $x_i$ with $i\in I_K$. Let $n$ be a natural number greater than $2\cdot\max\{|x_i|:\ x\in\varphi(\supp\calb), i\in I_K\}$. Let $N$ be the subgroup of $K$ generated by elements of the form $((0,0,\dots),(0,\dots,0,\underset{r\text{th}}{n},0,\dots))$ for every $r\in I_K$. Let $\psi:K\to K/N$ be the quotient map. Clearly, $K/N$ is
$$\left(\bigoplus_{p\in\mathbf{P}}\int[p^\infty]^{(I_{K,p})}\right)\oplus(\rat/n\int)^{(I_K)}.$$
This is a countable abelian torsion group, thus there is an embedding $\nu: K/N\embeds \wtilde A$ by Remark~\ref{r.abelian_embeds}. Notice that $\nu\circ\psi\circ\varphi:\ \wtilde H\to \wtilde A$ is homomorphism that is injective on the set $\supp\calb\subseteq\wtilde H$ by the choices of $n$ and $N$. Thus there is a bijection $\vartheta:\nat\to\nat$ that extends the finite bijection $(\nu\circ\psi\circ\varphi|_{\supp\calb})^{-1}$. For such a $\vartheta$ we have $h_\vartheta(A)\in\calb$ because $H$ and $h_\vartheta(A)$ coincide on $\{1,\dots,k\}\times\{1,\dots,k\}$ (recall from Observation~\ref{o.induced_homeo} that $h_\vartheta$ is the induced homeomorphism). Since $\wtilde{h_\vartheta(A)}\cong\wtilde A$, we conclude that $\cali_A$ is dense in $\cala$.
\end{proof}

\begin{cor}\label{c.generic_abelian}
The isomorphism class $\cali_A$ is comeager in $\cala$.
\end{cor}

\section{Compact metrizable abelian groups}\label{s.cma_groups}

\subsection{Pontryagin duality}

Here we remind the reader to the notion and fundamental properties of Pontryagin duality. Let us denote the circle group by $\circle$. For a locally compact abelian group (hereafter LCA group) $G$ the dual group $\what G$ is the set of all continuous homomorphisms $\chi:G\to\circle$ with pointwise addition and the compact-open topology, and $\what G$ is also an LCA group (see \cite[Chapter~1]{RUDIN}). Also recall:

\begin{theorem}[Pontryagin Duality Theorem]
If $G$ is an LCA group, then its double dual $\widehatto{G}{\widehat G}$ is canonically isomorphic to $G$ itself.
\end{theorem}

For reference we list some well-known properties of Pontryagin duality here.

\begin{prop}\label{p.pontryagin}
For an LCA group $G$ the following hold:

(1) An LCA group $A$ embeds into $G$ $\iff$ $\what A$ is a quotient of $\what G$. \cite[Theorem~2.1.2.]{RUDIN}

(2) For a sequence $(A_i)_{i\in\nat}$ of compact abelian groups we have $\ds\what{\prod_{i\in\nat}A_i}=\bigoplus_{i\in\nat}\what A_i$. \cite[Theorem~2.2.3.]{RUDIN}

(3) $G$ is compact $\iff$ $\what G$ is discrete. \cite[Theorem~1.2.5.]{RUDIN}

(4) $G$ is finite $\iff$ $\what G$ is finite. (Follows from (3) and the Pontryagin Duality Theorem.)

(5) $G$ is compact metrizable $\iff$ $\what G$ is discrete countable. \cite[Corollary on page~96]{MORRIS}

(6) $G$ is compact torsion-free $\iff$ $\what G$ is discrete divisible. \cite[Theorem~31.]{MORRIS}

(7) $G$ is compact totally disconnected $\iff$ $\what G$ is a discrete torsion-group. \cite[Corollary~1. on page~99]{MORRIS}

(8) The dual of $\circle$ is $\int$. \cite[pages 12-13.]{RUDIN}
\end{prop}

\subsection{The space of compact metrizable abelian groups}\label{s.univ_cma}

We view $\circle$ as the additive group $\real/\int$, that is, the real numbers modulo 1, and we use the usual metric $d_\circle$. Clearly, the circle group is a compact metrizable abelian group (hereafter CMA group). Let $\TT$ denote the countably infinite direct product of $\circle$ with itself. Again, it is a CMA group and every closed subgroup of $\TT$ is a CMA group as well. We equip $\TT$ with the metric
$$d(x,y)=\sum_{n=1}^\infty\frac{1}{2^n}d_\circle(x(n),y(n)).$$
Note that by Proposition~\ref{p.pontryagin} (2) and (8) the dual of $\TT$ is $F\defeq\int^{(\nat)}$. By (1) a CMA group $G$ is embeddable into $\TT$ if and only if the countable discrete abelian group $\what G$ is a quotient of $F$. Since $F$ is the free abelian group on countably infinitely many generators, every countable discrete abelian group is a quotient of $F$, hence every CMA group is embeddable into $\TT$. Or equivalently, every CMA group is isomorphic (as a topological group) to some element of $\cals(\TT)$ (recall Definition~\ref{d.compact_subgroups}).

Thus we may view $(\cals(\TT),d_H)$ as \textit{the space of CMA groups} (here $d_H$ is the Hausdorff metric induced by $d$). Now we discuss a useful observation about finite groups in $\cals(\TT)$.

\begin{defi}\label{d.supp}
For a group $G\in\cals(\TT)$ we define the \textbf{support} of $G$ by
$$\supp (G)\defeq\{n\in\nat:\ \exists g\in G\ (g(n)\neq 0)\}.$$
\end{defi}

\begin{notation}
Let $\mathcal{F}$ denote the set of finite groups in $\cals(\TT)$ with finite support.
\end{notation}

\begin{prop}\label{p.finite_dense}
$\calf$ is dense in $\cals(\TT)$.
\end{prop}

\begin{proof}
Fix some $\eps>0$ and $K\in \cals(\TT).$  Since $\pi_{[n]}(K)\overset{d_H}{\to}K$ (as $n\to\infty$, recall Notation~\ref{n.projection}), we may assume $K\leq\circle^N\times\{0\}\times\{0\}\times\ldots$ for some $N\in\nat$. We abuse notation and use simply $\circle^N$ instead of $\circle^N\times\{0\}\times\{0\}\times\ldots$. We work with the compatible metric $d_{\circle^N}(x,y)\defeq\sum_{i=1}^N d_\circle(x(i),y(i))$.

Let $x_1,x_2\ldots x_k$ be a finite $\frac{\eps}{2}$-net in $K$, and let us fix an integer $M$ such that $M>\frac{2Nk}{\eps}$.

We need a simultaneous version of Dirichlet's Approximation Theorem \cite[page~27]{SCHMIDT}.

\begin{theorem}\label{t.dirichlet}
Let $\alpha_1,\ldots \alpha_d$ be real numbers and let $Q$ be a positive integer. Then there exist $p_1,\ldots p_d,q\in\int$, $1\leq q\leq Q^d$, such that $|\alpha_i-\frac{p_i}{q}|\leq\frac{1}{qQ}$ for all $1\leq i\leq d$.
\end{theorem}

We apply this theorem with the parameters $d=kN$ and $Q=M$ to the real numbers $x_1(1),\ldots x_1(N),\ldots x_k(1),\ldots x_k(N)$. Thus we get integers $p_1^1,\ldots p_N^1,\ldots p_1^k,\ldots p_N^k$ and $q$ with $1\leq q\leq M^{kN}$ such that

$$\left|x_i(j)-\frac{p_j^i}{q}\right|\leq\frac{1}{qM}<\frac{\eps}{2qkN}.$$

The last inequality follows from our assumption on $M$.

This means that $d_\circle\left(x_i(j),\frac{p_j^i}{q}\right)<\frac{\eps}{2qkN}$ holds on the circle for each $1\leq i\leq k$ and $1\leq j\leq N$.
Let $L$ be the generated subgroup $\left\langle\left( \frac{p_1^1}{q},\ldots,\frac{p_N^1}{q}\right),\ldots,\left( \frac{p_1^k}{q},\ldots,\frac{p_N^k}{q}\right) \right\rangle\leq\circle^N$. Note that $L$ is finite.

We claim that $d_H(L,K)<\eps$. First we prove that $L_\eps\supseteq K$ (recall that $L_\eps$ is the $\eps$-neighborhood of $L$). This holds since $x_1,x_2,\ldots, x_k$ was an $\frac{\eps}{2}$-net in $K$ and 

$$d_{\circle^N}\left(x_i,\left( \frac{p_1^i}{q},\ldots,\frac{p_N^i}{q}\right)\right)=\sum_{j=1}^N d_{\circle}\left(x_i(j), \frac{p_j^i}{q}\right)<\frac{\eps}{2qkN} \cdot N=\frac{\eps}{2qk}\leq\frac{\eps}{2}.$$ 

For $K_\eps\supseteq L$ notice that every element of $L$ is of the form
$$\sum_{i=1}^k l_i\cdot\left( \frac{p_1^i}{q},\ldots,\frac{p_N^i}{q}\right),$$

with $l_i\in\int$, $0\leq l_i<q$ for each $1\leq i\leq k$, thus

$$d_{\circle^N}{\left(\sum_{i=1}^k l_i x_i, \sum_{i=1}^k l_i\cdot\left( \frac{p_1^i}{q},\ldots,\frac{p_N^i}{q}\right)\right)}\leq\sum_{i=1}^k\sum_{j=1}^N l_i\cdot d_{\circle}{\left(x_i(j), \frac{p_j^i}{q}\right)}<k\cdot N\cdot q\cdot\frac{\eps}{2qkN} =\frac{\eps}{2},$$
which proves the claim since $\sum_{i=1}^k l_i x_i\in K$.
\end{proof}

\subsection{The comeager isomorphism class}

\begin{theorem}\label{t.isomclass_abelian}
There exists a generic compact metrizable abelian group. That is, there exists a comeager isomorphism class in $\cals(\TT)$. Namely, it is the isomorphism class of
$$Z=\prod_{p\in \mathbf{P}}(Z_p)^\nat,$$
where $\textbf{P}$ is the set of prime numbers and $Z_p$ is the group of $p$-adic integers.
\end{theorem}

\begin{remark}\label{r.dualZ}
In Remark~\ref{r.unique_abelian} we showed that $\wtilde A=\bigoplus_{p\in P}\int[p^\infty]^{(\nat)}$ is the unique, up to isomorphism, countable abelian group that is divisible, torsion and contains every finite abelian group up to isomorphism. It is well-known that the additive group $Z_p$ of $p$-adic integers is the Pontryagin dual of $\int[p^\infty]$ (see \cite[25.2]{HEWITT_ROSS}); therefore it follows by Proposition~\ref{p.pontryagin} (2) that $Z$ is the Pontryagin dual of $\wtilde A$. Hence by Proposition~\ref{p.pontryagin} (1), (4), (6) and (7) $Z$ is the unique CMA group (up to isomorphism) such that it is torsion-free, totally disconnected and every finite abelian group occurs as its quotient.

Thus for proving Theorem~\ref{t.isomclass_abelian} it suffices to verify that these three properties hold for the generic $K\in\cals(\TT)$, which will be done in the following three lemmas.
\end{remark}

\begin{lemma}\label{l.generic_finite_quotient}
For the generic $K \in \cals(\TT)$ every finite abelian group occurs as a quotient of $K$.
\end{lemma}
\begin{proof}
Let 
$$\calb_1:=\{K \in \cals(\TT): \text{every finite abelian group $A$ is a quotient of $K$}\}.$$
Recall that for compact groups quotients and continuous homomorphic images coincide, hence
$$\calb_1=\bigcap_{A \text{ is finite abelian}}\{K\in \cals(\TT):\ \exists\ \varphi:K \to A\text{ surjective continuous homomorphism}\}.$$
Since there are countably many finite abelian groups up to isomorphism, it sufficies to show that for a fixed finite abelian group $A$ the following subset of $\cals(\TT)$ is dense open:
$$\calb_A:=\{K\in \cals(\TT):\ \exists\ \varphi:K \to A\text{ surjective continuous homomorphism}\}.$$
\textbf{Claim 1.} The set $\calb_A$ is open.

Let $\{a_1,a_2,\ldots, a_n\}$ be the underlying set of $A$. For a fixed $K\in \calb_A$ let $\varphi:K\to A$ be a surjective continuous homomorphism and let $H:=\kernel\varphi$. Then we have a partition $K=\bigcup_{i=1}^n (k_i+H)$, where $k_i\in K$ with $\varphi(k_i)=a_i$. Note that $k_i+H$ is compact for each $i$ and let 
$$\delta:= \min\{\dist(k_i+H, k_j+H):\ 1\leq i,j\leq n,\ i\neq j\}>0.$$
We claim that $B_{d_H}{\left(K,\frac{\delta}{4}\right)}\subseteq \calb_A$. To show this, for any $L\in B_{d_H}{\left(K,\frac{\delta}{4}\right)}$ we have to find a surjective continuous homomorphism $\psi:L\to A$. By the choice of $\delta$, if we define $L_i:=L\cap (k_i+H)_{\frac{\delta}{4}}$, then $\bigcup_{i=1}^n L_i$ is a decomposition of $L$ into disjoint nonempty open sets. Let us define $\psi$ as 
$\psi(x)=a_i$ if and only if $x\in L_i$. The continuity of $\psi$ is clear and the surjectivity follows from the fact that each $L_i$ is nonempty. Now we have to show that $\psi$ is a homomorphism.

Pick any $l_i\in L_i$ and $l_j\in L_j$. Then there exist $x_i\in k_i+H$ and $x_j\in k_j+H$ such that $d_{\TT}(x_i,l_i)<\frac{\delta}{4}$, $d_{\TT}(x_j,l_j)<\frac{\delta}{4}$, $\psi(l_i)=\varphi(x_i)=a_i$ and $\psi(l_j)=\varphi(x_j)=a_j$. Note that
$$d_{\TT}(x_i+x_j, l_i+l_j)\leq d_{\TT}(x_i+x_j,x_i+l_j)+d_{\TT}(x_i+l_j,l_i+l_j)=d_{\TT}(x_j,l_j)+d_{\TT}(x_i,l_i)<\frac{\delta}{2},$$
therefore $l_i+l_j\in (x_i+x_j+H)_{\frac{\delta}{2}}$. Now it follows from $L\subseteq K_{\frac{\delta}{4}}$ and the choice of $\delta$ that $l_i+l_j\in (x_i+x_j+H)_{\frac{\delta}{4}}$. Thus, by the definition of $\psi$, we have $\psi(l_i+l_j)=\varphi(x_i+x_j)=a_i+a_j$. We conclude that $\psi$ is a homomorphism, which completes the proof of Claim 1.

\textbf{Claim 2.} The set $\calb_A$ is dense in $\cals(\TT)$.

By Proposition~\ref{p.finite_dense} it sufficies to approximate elements of $\calf$. Fix $F\in \calf$ and $\eps>0$. Let $N\in\nat$ be such that $\frac{1}{2^N}<\eps$ and $\{1,\ldots,N\}\supseteq\supp(F)$. 
Since $A$ is a CMA group, there is an embedding
$$\varphi:A\embeds\underbrace{\{0\}\times\ldots\times\{0\}}_{N\text{ times}}\times \circle\times \circle\ldots.$$
Now the subgroup $H$ of $\TT$ generated by $F$ and $\varphi(A)$ is isomorphic to $F\times A$. Therefore, we have a natural surjective continuous homomorphism $H\to A$ and also $d_H(H,F)<\eps$ because $\pi_{[N]}(H)=\pi_{[N]}(F)$ and $\frac{1}{2^N}<\eps$.
\end{proof}

\begin{remark}\label{r.cma_not_simple}
It follows immediately from Lemma~\ref{l.generic_finite_quotient} that the generic $K\in\cals(\TT)$ is not simple.
\end{remark}

\begin{lemma}\label{l.generic_torsion_free}
The generic $K \in \cals(\TT)$ is torsion-free.
\end{lemma}

\begin{proof}
Let 
$$\calb_2:=\{K \in \cals(\TT): K \text{ is torsion-free} \}.$$
We will show that $\calb_2$ is dense $G_\delta$ in $\cals(\TT)$.

\textbf{Claim 1.} The set $\calb_2$ is $G_\delta$ in $\cals(\TT)$.

Let us define  $f_n:\TT\rightarrow \TT$ as $f_n: x\mapsto n\cdot x$. Clearly, $f_n$ is continuous and 
$$\calb_2=\left\{K\in \cals(\TT):\ \forall n,k\in \nat\  \left(0\notin f_n{\left[K\setminus B{\left(0,\tfrac{1}{k}\right)}\right]}\right)\right\}=$$
$$=\bigcap_{n,k\in\nat} \underbrace{\left\{K\in \cals(\TT):\ K\cap \underbrace{B{\left(0,\tfrac{1}{k}\right)}^c\cap f_n^{-1}(\{0\})}_{\text{closed}}=\emptyset\right\}}_{\text{open}}.$$
\textbf{Claim 2.} The set $\calb_2$ is dense in $\cals(\TT)$.

By Proposition~\ref{p.finite_dense} it suffices to approximate elements of $\calf$. Fix any $F\in \calf$ and $\eps>0$. Let $N\in\nat$ be such that $\frac{1}{2^N}<\eps$ and $\{1,\ldots,N\}\supseteq\supp(F)$. As in Theorem~\ref{t.isomclass_abelian}, let $Z=\prod_{p\in P}(Z_p)^\nat$, which is torsion-free (recall Remark~\ref{r.dualZ}). Since $Z$ is a CMA group, there is an embedding
$$\varphi:Z\embeds\underbrace{\{0\}\times\ldots\times\{0\}}_{N\text{ times}}\times \circle\times \circle\times\ldots.$$
As we have already observed in Remark~\ref{r.dualZ}, every finite abelian group is a quotient of $Z$. Since $Z$ and $\varphi(Z)$ are isomorphic, there is a surjective continuous homomorphism $\theta:\varphi(Z)\to F$.
Now consider the following map:
$$\Phi: F\times \varphi(Z)\to F\times F,\quad \Phi(f,x):=(f,\theta(x)).$$
(More precisely, instead of $F\times \varphi(Z)$ we should consider the subgroup of $\TT$ generated by $F$ and $\varphi(Z)$ as in the end of the proof of Lemma~\ref{l.generic_finite_quotient}. Here we avoid such rigor for notational simplicity.) Clearly, $\Phi$ is a surjective continuous homomorphism. Then for the diagonal subgroup $\tilde{F}=\{(f,f):f \in F\}\leq F\times F$ we have $K:=\Phi^{-1}(\tilde{F})\in\cals(\TT)$. We claim that $K$ is torsion-free and $d_H(K,F)<\eps$.

Assume that $(f,x)\in K\setminus\{(0,0)\}$ (with $f\in F$ and $x\in\varphi(Z)$) and $n(f,x)=(nf,nx)=(0,0)$ with some $n\in\nat$. Since $x\in \varphi(Z)$, and $\varphi(Z)$ is torsion-free, we get $x=0$. On the other hand, $(f,x)\in K$ means $\theta(x)=f$, hence we have $f=0$. Thus $K$ is torsion-free.

Note that $\pi_{[N]}(K)=\pi_{[N]}(F)$, therefore $d_H(F,K)<\eps$. This completes the proof of Claim 2.
\end{proof}

\begin{lemma}\label{l.generic_tot_disconnected}
The generic $K \in \cals(\TT)$ is totally disconnected.
\end{lemma}

\begin{proof}
Let 
$$\calb_3:=\{K \in \cals(\TT):\ K \text{ is totally disconnected} \}.$$
We need to show that for the generic $K\in\cals(\TT)$ the singleton $\{0\}$ is a connected component. Notice that whenever $\varphi:K\to F$ is a continuous homomorphism to a finite abelian group, the connected component of $0$ is a subset of $\kernel\varphi$. Let us define
$$\calb_3^{'}:=\bigcap_{n\in\nat}\underbrace{\left\{K \in \cals(\TT):\ \begin{array}{ll}
    \text{there exists a finite abelian group $F_n$ and a continuous} \\
    \text{ homomorphism $\varphi_n:K\to F_n$ with} \diam(\kernel(\varphi_n))<\frac{1}{n}
\end{array}
\right\}}_{U_n}.$$
Now we have $\calb_3^{'}\subseteq\calb_{3}$, hence it suffices to prove that $\calb_3^{'}$ is dense $G_\delta$ in $\cals(\TT)$. The denseness follows directly from Proposition~\ref{p.finite_dense} and the fact that every finite group is in $\calb_3^{'}$.

\textbf{Claim.} The set $U_n$ is open.

Pick some $K\in U_n$. Fix a finite abelian group $F_n$ and a continuous homomorphism $\varphi_n:K\to F_n$ with $\diam(\kernel(\varphi_n))<\frac{1}{n}$. As in the proof of Claim 1 of Lemma~\ref{l.generic_finite_quotient}, for any given $\eps>0$ if $L\in\cals(\TT)$ is suitably close to $K$, then $\varphi_n$ gives rise to a continuous homomorphism $\psi_n:L\rightarrow F_n$ with $\kernel(\psi_n)\subseteq \left(\kernel(\varphi_n)\right)_\eps$. Then $\eps:=\frac{1}{2}\left(\frac{1}{n}-\diam(\kernel(\varphi_n))\right)$ witnesses that $U_n$ is open.
\end{proof}

\subsection{Countable abelian groups revisited}\label{ss.revisited}

In Section~\ref{ss.abelian_case} we introduced a  Polish space $\cala$ of the (infinite) countable abelian groups. Now we present a different approach and discuss its relation to both $\cala$ and $\cals(\TT)$. As in Section~\ref{ss.abelian_case}, we do not use additive notation \textbf{in elements of} $\cala$.

\begin{notation}
Recall that $\zdo$ is the countably infinite direct sum of $\int$ with itself. Let $\cala'$ denote the set of all subgroups of $\zdo$ with the subspace topology inherited from the Cantor space $2^{\zdo}$.
\end{notation}

It is easy to check that the subspace $\cala'$ is closed in $2^{\zdo}$, thus $\cala'$ is Polish. (Compare this to \cite[page 16-17]{GOLDBRING}.) Since every countable abelian group is a quotient of $\zdo$, we may use $\cala'$ to define generic properties among countable abelian groups:

\begin{defi}\label{d.generic_property}
A group property $\calp$ is $\cala'$-\textbf{generic} if the set
$$\{H\in\cala':\ \zdo/H\text{ is of property }\calp\}$$
is comeager in $\cala'$. (Here by a group property $\calp$ among countable abelian groups we mean an isomorphism-invariant subset of the set of all quotients of $\int^{(\nat)}$.)
\end{defi}

The following theorem is a special case of \cite[Theorem~14.]{FISHER_GARTSIDE_1}.

\begin{theorem}[Fisher--Gartside]\label{t.fisherhom}
There is a canonical homeomorphism between $\cala'$ and $\cals(\TT)$.
\end{theorem}

\begin{proof}[Sketch of proof]
\begin{notation}
Let $H$ be a fixed locally compact abelian group, and $\what{H}$ its Pontryagin dual. If $G$ is a closed subgroup of $H$, then the $\textbf{annihilator}$ of $G$ is defined as:
$$\Ann(G):=\{\chi\in\what{H}:\ \forall a \in G\ (\chi(a)=0)\}.$$
Notice that $\Ann(G)$ is subgroup of $\what{H}$.
\end{notation}

Recall that any CMA group $G$ can be embedded into $\TT$, hence the annihilator of $G$ is a subgroup of $\zdo$, that is, $\Ann(G)\in \cala'$. Thus the canonical map, which turns out to be a homeomorphism between $\cala'$ and $\cals(\TT)$, is $\Ann$ itself.
\end{proof}

It is well-known that for a subgroup $G\leq \TT$ the dual $\what G$ is isomorphic to $\what{\TT}/\Ann(G)$ (see \cite[Theorem~2.1.2.]{RUDIN}). We conclude that the property of being isomorphic to the group $\bigoplus_{p\in P}\int[p^\infty]^{(\nat)}$, which is the dual of $Z$ in Theorem~\ref{t.isomclass_abelian}, is $\cala'$-generic. Recall that the isomorphism class of this group is comeager in $\cala$. It is natural to ask whether there is a simple topological connection between $\cala$ and $\cala'$ that explains this phenomenon.

\begin{notation}\label{n.supp_zdo}
For $x\in\zdo$ let
$$\supp (x)\defeq\{n\in\nat:\ x(n)\neq 0\}.$$
\end{notation}

\begin{notation}
Let $e_i$ be the following element of $\zdo$: $e_i(k)=1$ if $k=i$ and $e_i(k)=0$ otherwise.
\end{notation}

We will show that $\cala$ naturally embeds into $\cala'$.
Let us pick any $A \in \cala$. Recall that elements of $\zdo$ are $\nat\to\int$ functions with finite support. Now let
$$\Phi(A)\defeq\left\{x\in\zdo:\ \prod_{n\in\nat}n^{x(n)}=1\text{ holds in }\wtilde A\right\}.$$
Clearly, $\Phi(A)$ is a subgroup of $\zdo$, that is, $\Phi(A)\in\cala'$.

\begin{prop}\label{l.2abel_hom}
The map $\Phi:\cala\to\cala'$ is an embedding.
\end{prop}

\begin{proof}
\textbf{Injectivity.} Pick distinct elements $A,B\in \cala$. If we use the original notations introduced in the beginning of Subsection~\ref{ss.multip_tables}, this means that there exist $k,l,m\in\nat$ such that $m=A(k,l)\neq B(k,l)$. By the definition of $\Phi$ it follows that $e_k+e_l-e_m\in \Phi(A)$ and $e_k+e_l-e_m\notin\Phi(B)$. 

\textbf{Continuity and openness.} Sets of the form
$$U_x=\{A'\in \cala': x\in A'\}\quad \text{ and }\quad V_x=\{A'\in \cala': x\notin A'\}$$
with $x\in\zdo$ constitute a subbasis in $\cala'$.
By the definition of $\Phi$ we have
$$\Phi^{-1}(U_x)=\left\{A\in \cala:\ \prod_{n\in\nat} n^{x(n)}=1 \text{ in } \wtilde{A}\right\},$$
and
$$\Phi^{-1}(V_x)=\left\{A\in \cala:\ \prod_{n\in\nat} n^{x(n)}\neq 1 \text{ in } \wtilde{A}\right\}.$$
By Proposition~\ref{p.basis2} these sets are open, hence $\Phi$ is continuous. Again, by Proposition~\ref{p.basis2} the set $\{\Phi^{-1}(U_x):\ x\in\zdo\}$ is a subbasis for $\cala$. Thus $\Phi(\Phi^{-1}(U_x))= U_x\cap\Phi(\cala)$ and the injectivity of $\Phi$ shows that $\Phi$ is open (onto its range).
\end{proof}

\begin{prop}
The set $\Phi(\cala)$ is nowhere dense in $\cala'$.
\end{prop}

\begin{proof}
Observe that $e_i\notin A'$ if $i\geq 2$ and $A'\in \Phi(\cala)$, therefore we have $\Phi(\cala)\subset \bigcap_{i=2}^\infty V_{e_i}$. Then it suffices to prove that $\bigcup_{i=2}^\infty U_{e_i}$ is dense in $\cala'$. Pick any nonempty basic clopen set $U\defeq U_{x_1}\cap\ldots\cap U_{x_n}\cap V_{y_1}\cap\ldots\cap V_{y_k}$ in $\cala'$. Let $N\in\nat$ be such that $N\notin(\bigcup_{i=1}^n \supp(x_i))\cup(\bigcup_{j=1}^k\supp(y_j))\cup\{1\}$. Now it is easy to check that $U\cap U_{e_N}\neq\emptyset$.
\end{proof}

\textbf{Summary.} The two approaches ($\cala$ and $\cala'$) give the same comeager isomorphism class: the class of $\bigoplus_{p\in P}\int[p^\infty]^{(\nat)}$. Although there is a natural embedding $\Phi:\cala\embeds\cala'$, this does not explain the coincidence since $\Phi(\cala)$ is nowhere dense in $\cala'$.

\section{Compact metrizable groups}\label{s.compact_metrizable_groups}

In this section, we study generic properties of compact metrizable groups. We will see a more obscure picture than in the abelian case, however, we believe our results constitute a promising starting point for interesting future research.

We prove that both in terms of connectedness and torsion elements the generic group is heterogeneous.

\subsection{The space of compact metrizable groups}

It is well-known that every compact group $G$ can be embedded into a direct product $\prod_{j\in J}U(n_j)$, where $n_j\in\nat$ and $U(n_j)$ is a unitary group \cite[Corollary~2.29.]{HOFMANN_MORRIS}. This relies on the fact that unitary representations separate points of $G$. If $G$ is metrizable, then it is easy to check that even countably many representations separate points of $G$, therefore every compact metrizable group can be embedded into a countable direct product of unitary groups.

Also, note that $U(n)$ is embeddable into the special unitary group $SU(k)$ for $k>n$. We conclude that every compact metrizable group is embeddable into $\su\defeq\prod_{n=1}^\infty SU(n)$.

Let $d_n$ be a biinvariant metric on $SU(n)$ with $d_n\leq 1$ for every $n\in\nat$. We shall use the metric $d(x,y)\defeq\sum_{n=1}^\infty\frac{1}{2^n}d_n(x(n),y(n))$ on $\su$ and the Hausdorff metric $d_{n,H}$ (resp. $d_H$) induced by $d_n$ (resp. $d$) on $\cals(SU(n))$ (resp. $\cals(\su)$).

Clearly, every closed subgroup of $\su$ is a compact metrizable group. Thus we may view $(\cals(\su),d_H)$ as \textit{the space of compact metrizable groups}.

\begin{obs}\label{o.hconv}
Let $G\in\cals(\su)$ and $G_n\in\cals(\su)$ for every $n\in\nat$. Notice that if $\pi_{[n]}(G_n)=\pi_{[n]}(G)$ for every $n\in\nat$, then $G_n\hconv G$.
\end{obs}

\subsection{Approximating \texorpdfstring{$SU(n)$}{}}

In this subsection we prove a lemma for later use. It also raises Question~\ref{q.approx_proper_sgp} and Question~\ref{q.isolated}, which are quite interesting in their own.

\begin{lemma}\label{l.sun_proper_subgroups}
The group $SU(n)$ cannot be approximated by closed proper subgroups. That is, $SU(n)$ is an isolated point in $\cals(SU(n))$.
\end{lemma}

To prove this we need the following three theorems.

\begin{theorem}[{Cartan, \cite[Theorem~20.10]{LEE}}]\label{t.cartan}
 A closed subgroup of a Lie group is a Lie subgroup.
\end{theorem}

\begin{theorem}[{Mostow, \cite[Theorem~7.1.]{MOSTOW}}]\label{t.nonconjugate}
In a compact Lie group, any set of connected Lie subgroups whose normalizers are mutually non-conjugate is finite.
\end{theorem}

The following is a special case of  \cite[Theorem~2.]{TURING}.

\begin{theorem}[Turing]\label{t.turing}
Let $G$ be a connected Lie group with a compatible left invariant metric $d$. If $G$ can be approximated by finite subgroups in the sense that for every $\eps>0$ there is a finite subgroup $F\leq G$ that is an $\eps$-net in $G$, then $G$ is compact and abelian.
\end{theorem}

\begin{remark}
Turing formulates his theorem without the assumption of connectedness, but this condition \emph{cannot be omitted} as the example of a finite nonabelian discrete group shows. Actually, the proof of Turing works (only) if the group $G$ is connected. 
\end{remark}

\begin{defi}\label{d.identity_component}
The \textbf{identity component} of a topological group $G$ is the connected component of its identity element. We denote it by $G^\circ$. It is well-known that $G^\circ$ is a closed normal subgroup of $G$.
\end{defi}

\begin{proof}[Proof of Lemma~\ref{l.sun_proper_subgroups}]
The trivial group $SU(1)$ has no proper subgroup, hence we may assume $n\geq 2$. Suppose that $(K_i)_{i\in \nat}\leq SU(n)$ is a sequence of proper compact subgroups such that $K_i\overset{d_{n,H}}{\longrightarrow} SU(n)$. Notice that $(K_i)^\circ$ is closed and therefore, by Theorem~\ref{t.cartan}, it is a connected Lie subgroup of $SU(n)$ for every $i\in\nat$.

Consider the sequence of normalizers $N_i\defeq N_{SU(n)}((K_i)^\circ)$. By Theorem~\ref{t.nonconjugate} we may assume that the $N_i$ are mutually conjugate. Since $(K_i)^\circ\nsub K_i$, we have $K_i\leq N_i\leq SU(n)$. Thus $N_i\overset{d_{n,H}}{\longrightarrow}SU(n)$ as well. On the other hand, $d_n$ is biinvariant, hence the distance $d_H(SU(n), N_i)$ is the same for each $i\in\nat$. We conclude that $N_i=SU(n)$ for every $i\in\nat$. In other words, every $(K_i)^\circ$ is normal in $SU(n)$.

The Lie algebra of a connected normal subgroup of $SU(n)$ is an ideal of the Lie algebra $\mathfrak{su}(n)$ (see \cite[page~128]{HELGASON}). Since $\mathfrak{su}(n)$ is simple (see \cite[page~13]{PFEIFER}), we conclude that every $(K_i)^\circ$ is the trivial group. Therefore, $K_i$ is finite for every $i\in\nat$.

By Theorem~\ref{t.turing} this contradicts the fact that $SU(n)$ is a connected nonabelian group (see \cite[Proposition~13.11]{HALL} for the connectedness).
\end{proof}

Although for our purposes Lemma~\ref{l.sun_proper_subgroups} suffices, it raises natural questions:

\begin{question}\label{q.approx_proper_sgp}
Which compact Lie groups can be approximated by proper subgroups in the Hausdorff metric?
\end{question}

Or more generally:

\begin{question}\label{q.isolated}
Let $G$ be a Lie group. Which compact subgroups $K$ of $G$ can be approximated by other compact subgroups of $G$ that are non-conjugate to $K$? (It follows easily that in a compact connected Lie group $G$ any nonnormal compact subgroup can be approximated by its conjugates.)
\end{question}

\begin{remark}
We remark that Question~\ref{q.isolated} have been recently answered in \cite{CSIKOS}.
\end{remark}

\subsection{Generic topological properties}

We will need the following well-known theorems. 

\begin{theorem}\label{t.nss}
\cite[Chapter~II, Exercise~B.5.]{HELGASON} Every Lie group has the ''no small subgroup'' property. That is, in a Lie group $G$ there exists a neighborhood $U$ of the identity such that $U$ contains no nontrivial subgroups of $G$.
\end{theorem}

\begin{theorem}\label{t.vandantzig}
\cite[Theorem~1.34.]{HOFMANN_MORRIS} For a locally compact group $G$ the following are equivalent:

(1) The group $G$ is totally disconnected.

(2) The identity of $G$ has a neighborhood basis consisting of compact open subgroups.
\end{theorem}

\begin{lemma}\label{l.tot_disconnected_proj}
If $G\in \cals(\su) $ is totally disconnected, then $\pi_n(G)$ is finite for each $n\in\nat$.
\end{lemma}

\begin{proof}
Fix $n\in\nat$. By Theorem~\ref{t.nss} there is a neighborhood $U$ of the identity of $SU(n)$ that contains no nontrivial subgroups of $SU(n)$. By Theorem~\ref{t.vandantzig} and the continuity of $\pi_n$ we can pick a compact open subgroup $H$ of $G$ such that $\pi_n(H)\subseteq U$ and thus $\pi_n(H)=\{1_{SU(n)}\}$. Also notice that the index of $H$ is finite since $H$ is open, therefore $\pi_n(G)$ is finite.
\end{proof}

\begin{theorem}\label{t.full_projection}
For the generic $G\in\cals(\su)$ there are infinitely many $n\in \nat$ such that $\pi_n(G)=SU(n)$.
\end{theorem}

\begin{proof}
Let 
$$\calu:=\{G\in \cals(\su):\ \forall k\in \nat\ \exists n\geq k \ (\pi_n(G)=SU(n))\}=$$
$$=\bigcap_{k=1}^\infty\bigcup_{n=k}^\infty\underbrace{\{G\in\cals(\su):\ \pi_n(G)=SU(n)\}}_{\calu_n}.$$
It follows from Lemma~\ref{l.continuous_image} and Lemma~\ref{l.sun_proper_subgroups} that each $\calu_n$ is open in $\cals(\su)$, hence $\calu$ is $G_\delta$. Thus it sufficies to prove that $\calu$ is dense. Pick any group $G\in \cals(\su)$.  Then, with a slight abuse of notation,
$$(\pi_{[m]}(G)\times SU(m+1)\times SU(m+2)\times\ldots)\hconv G$$
shows that $G$ can be approximated by elements of $\calu$.
\end{proof}

\begin{cor}\label{t.disconnected-nowheredense}
Totally disconnected groups form a nowhere dense set in $\cals(\su)$.
\end{cor}

\begin{proof}
We have shown in the proof of Theorem~\ref{t.full_projection} that for any $k\in\nat$ the set $\bigcup_{n=k}^\infty\calu_n$ is dense open. On the other hand, by Lemma~\ref{l.tot_disconnected_proj} it does not contain totally disconnected groups for $k\geq 2$.
\end{proof}
 
\begin{cor}\label{c.id_comp_nontrivial}
For the generic $G\in\cals(\su)$ the identity component $G^\circ$ is nontrivial.
\end{cor}

\begin{theorem}\label{t.quotient_cantor}
For the generic $G\in\cals(\su)$ the quotient $G/G^\circ$ is homeomorphic to the Cantor set.
\end{theorem}

\begin{proof}
Since $G/G^\circ$ is a totally disconnected metrizable (see \cite[Theorem~4.5]{HAAGERUP}) compact group, it suffices to show that for the generic $G\in\cals(\su)$ it is not finite. Let $n$ be a fixed positive integer. Let 
$$\calu_n\defeq\{G\in\cals(\su):\ |G:G^\circ|\geq n \}.$$
We will show that $\calu_n$ is open and dense, which completes the proof.

\textbf{Claim 1.} The set $\calu_n$ is dense.

Note that $SU(m)$ contains the cyclic group $C_m$ for each $m\in\nat$ (up to isomorphism). Also for any $G\in\cals(\su)$ we have, with an abuse of notation,
$$G_m\defeq\pi_{[m]}(G)\times C_{m+1}\times \{1\}\times \{1\}\times\ldots \overset{d_H}{\to} G$$
by Observation~\ref{o.hconv}. Here $G_m$ has at least $m+1$ connected components, hence $\calu_n$ is dense.

\textbf{Claim 2.} The set $\calu_n$ is open.

Pick any group $G\in\calu_n$. First we show that $G$ can be partitioned into finitely many but at least $n$ clopen cosets. It follows from Theorem~\ref{t.vandantzig} that the group $G/G^\circ$ can be partitioned into finitely many but at least $n$ clopen cosets. Since the natural homomorphism $G\to G/G^\circ$ is continuous, the inverse images of these clopen cosets form a suitable partition of $G$.

Now let $\delta$ be the minimal distance occuring between these clopen cosets of $G$. Then it is easy to see that any $L\in\cals(\su)$ with $d_H(L,K)<\frac{\delta}{2}$ has at least $n$ connected components.
\end{proof}

\begin{remark}\label{r.cm_not_simple}
It follows immediately from Corollary~\ref{c.id_comp_nontrivial} and Theorem~\ref{t.quotient_cantor} that the generic $G\in\cals(\su)$ is not simple. In particular, it is not algebraically closed by Remark~\ref{r.prop_of_ac}.
\end{remark}

\begin{cor}\label{c.cantor_x_conncomp}
The generic $G\in \cals(\su)$ is homeomorphic to $G^\circ\times C$, where $C$ is the Cantor set.
\end{cor}

\begin{proof}
By \cite[Corollary~10.38.]{HOFMANN_MORRIS} for every compact group $G$ the group $G^\circ\times G/G^\circ$ is homeomorphic to $G$, hence the corollary follows from Theorem~\ref{t.quotient_cantor}.
\end{proof}

The following useful fact is well-known. We present a proof for completeness.

\begin{lemma}\label{l.connected_component}
Let $G$ and $H$ be compact groups and suppose that $H$ is connected. Then for any surjective continuous homomorphism $\Phi:G\to H$ we have $\Phi(G^\circ)=H$.
\end{lemma}

\begin{proof}
Let $\varphi:G/G^\circ\to H/\Phi(G^\circ)$ be the natural continuous surjection, that is, $\varphi(gG^\circ):=\Phi(g)\Phi(G^\circ)$. We need to prove that $H/\Phi(G^\circ)$ is trivial.

Suppose that there is an open neighborhood $V$ of the identity of $H/\Phi(G^\circ)$ such that $V\neq H/\Phi(G^\circ)$. By Theorem~\ref{t.vandantzig} and the continuity of $\varphi$ there is a clopen subgroup $U\leq G/G^\circ$ such that $\varphi(U)\subseteq V$.

Recall that $\varphi$ is an open map because it is a surjective continuous homomorphism between compact groups. Therefore, $\varphi(U)\neq H/\Phi(G^\circ)$ is clopen, which contradicts the connectedness of $H/\Phi(G^\circ)$.
\end{proof}

This lemma allows us to state a slight strengthening of Theorem~\ref{t.full_projection}:

\begin{cor}\label{c.id_comp_full_proj}
For the generic $G\in\cals(\su)$ there are infinite many $k\in\nat$ such that $\pi_k(G^\circ)=SU(k)$.
\end{cor}

\begin{theorem}\label{t.notlie}
For the generic $G\in \cals(\su)$ the connected component $G^\circ$ is not Lie.
\end{theorem}

\begin{proof}
Fix $n\in\nat$ arbitrarily. It sufficies to prove that for the generic $G\in\cals(\su)$ its identity component $G^\circ$ is not an $n$-dimensional Lie group. By Corollary~\ref{c.id_comp_full_proj}, for the generic $G\in\cals(\su)$ we have $\pi_k(G^\circ)=SU(k)$ for infinitely many $k\in\nat$.

Suppose that $G^\circ$ is an $n$-dimensional Lie group, and fix some $k>n$ with $\pi_k(G^\circ)=SU(k)$. It is well-known that any continuous homomorphism between Lie groups is smooth. (See \cite{WANG} or \cite[Corollary~3.50.]{HALL}. Although this corollary is stated for matrix Lie groups, the proof works for any Lie group.) Thus $\pi_k|_{G^\circ}:G^\circ\rightarrow SU(k)$ is smooth. It is also surjective, therefore, by Sard's lemma \cite[Theorem~4.1]{SARD}, the differential of $\pi_k|_{G^\circ}$ is surjective at some point, which contradicts $k>n$.
\end{proof}

We summarize the topological results in the following theorem.

\begin{theorem}\label{t.compact_metriz_top}
The generic $G\in\cals(\su)$ is homeomophic to $G^\circ\times C$, where $C$ is the Cantor set and $G^\circ$ is a nontrivial connected group but not a Lie group.
\end{theorem}

\begin{proof}
Follows immediately from Corollary~\ref{c.cantor_x_conncomp} and Theorem~\ref{t.notlie}. 
\end{proof}

\subsection{Torsion elements}

Based on the abelian case one could suspect that the generic $G\in\cals(\su)$ is torsion-free. We will show that this is not true. The following proposition is a reformulation of Proposition~2.~c) in \cite[Chapter~IX, Appendix~I, Section~3]{BOURBAKI_7-9}.

\begin{prop}\label{p.bourbaki}
For any compact connected group $G$ there is a family $\{S_i:\ i\in I\}$ of simple connected compact Lie groups and a surjective continuous homomorphism $\prod_{i\in I}S_i\to [G,G]$ (where $[G,G]$ is the commutator subgroup of $G$), whose kernel is a totally disconnected, compact, central subgroup.
\end{prop}

In \cite[Chapter~IX, Appendix~I, Section~3]{BOURBAKI_7-9} the term \textit{almost simple} is used. However, Bourbaki's definition of an \textit{almost simple} Lie group agrees with what most modern authors call a connected simple Lie group, see \cite[Chapter~III, Section~9.8, Definition~3]{BOURBAKI_1-3}.

The preceding reference and the following proposition was communicated to us by Yves de Cornulier on the website mathoverflow.com. We remind the reader that every simple connected compact Lie group has finite center (see \cite[page~128]{HELGASON}). 

\begin{prop}\label{p.cc_torsion_free}
Every compact connected torsion-free group is abelian.
\end{prop}

\begin{proof}
Let $G$ be a compact connected torsion-free group and let $\{S_i:\ i\in I\}$ be the family of simple connected compact Lie groups provided by Proposition~\ref{p.bourbaki}. Also let $\varphi:S\defeq \prod_{i\in I}S_i\to [G,G]$ denote the surjective continuous homomorphism provided by Proposition~\ref{p.bourbaki}.

If $[G,G]=1$, then $G$ is abelian and we are done. Otherwise, $I$ is nonempty and we have $\kernel\varphi\subseteq Z(S)=\prod_{i\in I}Z(S_i)<S$ by Proposition~\ref{p.bourbaki}. We claim that $S$ contains a non-central torsion element. Pick $j\in I$ such that $S_j\setminus Z(S_j)\neq\emptyset$. Clearly, it suffices to prove that $S_j\setminus Z(S_j)$ contains a torsion element. We may assume that $S_j$ is infinite. Then it contains a nontrivial torus (recall that a every infinite compact Lie group contains a nontrivial torus \cite[Lemma~6.20.]{HOFMANN_MORRIS}) and thereby infinitely many torsion elements, while $Z(S_j)$ is finite. Now it follows that $G$ has a nontrivial torsion element, a contradiction.
\end{proof}

\begin{lemma}\label{l.torsion_free_sun}
For $n\geq 2$ the group $SU(n)$ does not occur as a quotient of a compact torsion-free group.
\end{lemma}

\begin{proof}
Suppose that $SU(n)$ occurs as a quotient of a compact torsion-free group $G$. By Lemma~\ref{l.connected_component} we may assume that $G$ is connected. Then by Proposition~\ref{p.cc_torsion_free} it must be abelian. Since $SU(n)$ is nonabelian for $n\geq 2$, this is a contradiction.
\end{proof}

\begin{theorem}\label{t.torsion}
In the generic $G\in\cals(\su)$ there are both torsion and nontorsion elements.
\end{theorem}

\begin{proof}
By Theorem~\ref{t.full_projection} for the generic $G\in\cals(\su)$ some $SU(k)$ is a quotient of $G$ with $k\geq 2$. Since $SU(k)$ contains a circle group, it contains elements of infinite order, and so does $G$. On the other hand, by Lemma~\ref{l.torsion_free_sun}, $G$ cannot be a torsion-free group either.
\end{proof}

\subsection{Free subgroups}

\begin{notation}\label{n.generated_free_group}
For a group $G$ and a number $m\in\nat$ let
$$
F(m,G):=\{ (g_1,...,g_m) \in G^m:\ g_1,...,g_m  \text{ freely generate a free group in }G \}.
$$
For any set $A$ let $(A)^m\defeq\{(x_1,\ldots,x_m)\in A^m:\ x_i\neq x_j \text{ if }i\neq j\}$.
\end{notation}

\begin{lemma}\label{l.free_in_sun}
For every $n\geq 3$ and $m\in\nat$ the set $F(m,SU(n))$ is comeager in $SU(n)^m$.
\end{lemma}

\begin{proof}
Fix $m \in \nat$ and a nonempty reduced word $W(x_1,...,x_m)$ in $m$ variables arbitrarily. Consider the analytic map
$$
f_W:SU(n)^n \to SU(n), \quad (g_1,...,g_m) \mapsto W(g_1,...,g_m),
$$
from a connected analytic manifold to an analytic manifold (see \cite{EPSTEIN}).

\textbf{Claim.} The compact set $f_W^{-1}(\identity_{SU(n)})$ has empty interior, hence it is nowhere dense.

Suppose $\interior(f_W^{-1}(\identity_{SU(n)}))\neq\emptyset$. Then the analytic map $f_W$ is constant on a nonempty open set. Therefore, it is constant on $SU(n)^m$ (see, for example, \cite{EPSTEIN}). On the other hand, it is well-known (and also follows from \cite{EPSTEIN}) that $SO(3)$ contains a free group of rank $m$, hence $SU(n)$ also contains a free group of rank $m$. This contradicts that $f_W$ is constant on $SU(n)^m$, which proves the claim.

Notice that we can write
$$
F(m,SU(n)) = \bigcap\limits_{\substack{W \textit{ is a nonempty reduced } \\ \textit{word in } m \textit{ variables}}} (SU(n)^m \setminus f_W^{-1}(\identity_{SU(n)})).
$$
By the claim, the right-hand side is a countable intersection of dense open sets, which concludes the proof. 
\end{proof}

We need the following classical result from descriptive set theory.

\begin{theorem}[Kuratowski--Mycielski, {\cite[Theorem~19.1]{KECHRIS}}]\label{t.kuratowski_mycielski}
Let $X$ be a metrizable space, let $m_1,m_2,\ldots$ be a sequence of natural numbers. If $R_i\subseteq X^{m_i}$ is comeager for each $i\in\nat$, then
$$\{K\in\calk(X):\ \forall i\in\nat\ ((K)^{m_i}\subseteq R_i)\}$$
is comeager in $K(X)$.
\end{theorem}

\begin{theorem}
For the generic $G\in\cals(\su)$ the generic $K\in\calk(G)$ freely generates a free group of rank continuum in $G$.
\end{theorem}

\begin{proof}
By Theorem~\ref{t.full_projection}, for the generic $G\in\cals(\su)$ there is $n\geq 3$ such that $\pi_n(G)=SU(n)$.  Fix such $G$ and $n$. We claim that the generic $K\in\calk(G)$ freely generates a free group of rank continuum in $G$.

Recall that by Theorem~\ref{t.surj_cont_hom_open}, the map $\pi_n:G\to SU(n)$ is continuous and open. It follows that the map
$\pi_n^m: G^m\to SU(n)^m$, $(g_1,\ldots,g_m)\mapsto (\pi_n(g_1),\ldots,\pi_n(g_m))$ is also continuous and open for every $m\in\nat$. Then, by Lemma~\ref{l.free_in_sun} and \cite[Proposition~2.8]{MELLERAY}, the inverse image $B_m\defeq(\pi_n^m)^{-1}(F(m,SU(n)))$ is comeager in $G^m$ for every $m\in\nat$. Observe that for every $m\in\nat$ and $(g_1,\ldots,g_m)\in B_m$ the elements $g_1,\ldots,g_m$ freely generate a free group in $G$ because, since $\pi_n$ is a homomorphism, a nontrivial relation of the $g_i$ in $G$ would give us a nontrivial relation of the $\pi_n(g_i)$ in $SU(n)$.

Now by Theorem~\ref{t.kuratowski_mycielski}, for the generic $K\in\calk(G)$ for every $m\in\nat$ we have $(K)^m\subseteq B_m$. Therefore, the generic $K\in\calk(G)$ freely generates a free group in $G$. Furthermore, by \cite[(8.8)~Exercise~i)]{KECHRIS}, the generic $K\in\calk(G)$ has continuum many elements, which concludes the proof.
\end{proof}






\section{Questions}

As a direct analogue of Definition~\ref{d.generic_property} we may introduce a notion of genericity among countable groups. Let $F_\infty$ denote the free group on countably infinitely many generators. It is straightforward to verify that the subspace $\caln\defeq\{A\subseteq F_\infty:\ A\triangleleft F_\infty\}$ is closed in $2^{F_\infty}$. Let us say that a property $\calp$ of countable groups is $\caln$-generic if the set $\{N\in\caln:\ F_\infty/N\text{ is of property }\calp\}$ is comeager in $\caln$.

\begin{problem}
It would be interesting to study $\caln$-genericity and explore its connections to results of Section~\ref{s.countably_infinite_discrete_groups}.
\end{problem}

For partial results in the abelian case see Subsection~\ref{ss.revisited}. The above notion of genericity was considered in \cite{OSIN} and \cite[page 16]{GOLDBRING}.

Although we obtained results on the structure of the generic compact metrizable group, a lot of basic questions remained open.

\begin{question}
Is there a comeager isomorphism class in $\cals(\su)$?
\end{question}

We also propose problems that seem more accessible.

\begin{problem}
Describe $G^\circ$ up to homeomorphism for the generic $G\in\cals(\su)$, if this is possible.
\end{problem}

\begin{remark}
One may think that $G^\circ$ is connected but behaves so wildly that it does not contain any nontrivial path. However, we claim that for the generic $G\in\cals(\su)$ there are many paths (even 1-parameter subgroups) in $G^\circ$. First, by Corollary~\ref{c.id_comp_full_proj}, we know that for the generic $G\in\cals(\su)$ for infinitely many $n\in\nat$ the projection $\pi_n:G^\circ \to SU(n)$ is surjective. Fix such $G$ and $n\geq 2$. Second, by \cite[Theorem~11.9]{HALL}, every $p \in SU(n)$ is contained in a (maximal) torus, so every point $p$ lies on a one parameter subgroup of $SU(n)$. Third, by \cite[Lemma~4.19]{hofmann2007lie}, every one parameter subgroup can be lifted among pro-Lie groups, that is, if $q:G \to H$ is a quotient morphism of pro-Lie groups and $\gamma:\mathbb{R} \to H$ is a one parameter subgroup, then there exists a one parameter subgroup $\tilde{\gamma}:\mathbb{R} \to G$ such that $\gamma = q \circ \tilde{\gamma}$. Thus, for every $p\in SU(n)$ there is a 1-parameter subgroup in $G^\circ$ that meets the set $\pi_n^{-1}(p)$.
\end{remark}

\begin{question}
What can we say about generic algebraic properties in $\cals(\su)$?
\end{question}

\section*{Acknowledgement}

We would like to express our gratitude to Yves de Cornulier, Bal{\'a}zs Csik{\'o}s and Krist{\'o}f Kanalas for their help. Yves de Cornulier contributed by providing Proposition~\ref{p.bourbaki} and Proposition~\ref{p.cc_torsion_free} on mathoverflow.com. Bal{\'a}zs Csik{\'o}s helped us several times when we dealt with Lie groups. Krist{\'o}f Kanalas worked with us while we developed the material of Section~\ref{s.countably_infinite_discrete_groups}.

This research was supported by the National Research, Development and Innovation Office -- NKFIH, grant no. 124749 and 146922.
The first author was also supported by the National Research, Development and Innovation Office -- NKFIH, grant no.~129211. The third author was also supported by the ÚNKP-21-3 New National Excellence Program of the Ministry for Innovation and Technology from the source of the National Research, Development and Innovation Fund. The fourth author was also supported by the National Research, Development and Innovation Office -- NKFIH, grant no.~129335. The fifth author was supported by the ÚNKP-21-1 New National Excellence Program of the Ministry for Innovation and Technology from the source of the National Research, Development and Innovation Fund. The last author was supported by Rényi Doctoral Fellowship.


\end{document}